\documentclass{amsart}

\usepackage{amssymb}

\usepackage{}
\usepackage{pdfsync}
\usepackage{enumerate}

\newtheorem{theorem}{Theorem}[section]
\newtheorem{thm}{Theorem}[section]
\newtheorem{lem}[theorem]{Lemma}

\theoremstyle{definition}
\newtheorem{definition}[theorem]{Definition}
\newtheorem{exam}[theorem]{Example}

\theoremstyle{remark}
\newtheorem{rem}[theorem]{Remark}

\newtheorem{prop}[theorem]{Proposition}
\newtheorem{cor}[theorem]{Corollary}

\newtheorem{notat}[theorem]{Notation}
\newtheorem{conv}[theorem]{CONVENTION}
\newtheorem{fact}[theorem]{Fact}

\newenvironment{displaytext}
{\begin{equation} \begin{minipage}{0.9\linewidth}\em}
{\end{minipage} \end{equation}}
\newcommand{\Add}{\mathop{{\rm Add}}\nolimits}
\newcommand{\add}{\mathop{{\rm add}}\nolimits}

\newcommand{\RMod}{R\text{\rm -Mod}}
\newcommand{\Rmod}{R{\text{\rm -mod}}}

\newcommand{\ModR}{\text{\rm Mod-}R}
\newcommand{\Ab}{\mathop{\text Ab}}
\newcommand{\modR}{\text{\rm mod-}R}

\newcommand{\Zg}{{\rm Zg}}

\newcommand{\Mod}{{\rm Mod}}

\newcommand{\card}{|R|+\aleph_0}

\newcommand{\D}{\rm{D}}
\newcommand{\pp}{{\rm{pp}}}
\newcommand{\qf}{{\rm{qf}}}

\newcommand{\ann}{{\rm{ann}}}
\newcommand{\im}{{\rm{im}\,}}

\newcommand{\cal}{\mathcal}
\newcommand{\mtx}{\mathfrak}
\renewcommand{\bar}{\overline}
\newcommand{\br}{\bar}

\renewcommand{\phi}{\varphi}
\renewcommand{\rho}{\varrho}
\newcommand{\al}{\alpha}

\newcommand{\K}{\cal K}
\newcommand{\Lo}{$\cal L$-\,$\omega$}

\renewcommand{\to}{\longrightarrow}

\newcommand{\Div}{\rm{\bf Div}}

\begin{document}
\dedicatory{To Mike}
\title{Mittag Leffler modules and definable subcategories}


\author{Philipp Rothmaler}
\address{CUNY, The Graduate Center, 365 Fifth Avenue, Room 4208, NY, NY 10016}
\curraddr{}
\email{philipp.rothmaler@bcc.cuny.edu}
\thanks{Supported in part by PSC CUNY Award 64595-00 42.}



\date{Contemporary Mathematics {\bf 730}, 2019, 171--196.\\ https://doi.org/10.1090/conm/730/14716}

\begin{abstract}
We study (relative) $\mathcal K$-Mittag-Leffler modules as was done in the author's habilitation thesis,  rephrase old, unpublished results in terms of definable subcategories, and present  newer ones, culminating in a characterization of countably generated $\cal K$-Mittag-Leffler modules.
\end{abstract}

\maketitle


\section{Introduction}

 We call a module $M$ \emph{countably pure-separable by pure-projectives} if every countable subset of $M$ is contained in a pure submodule of $M$ that is itself pure-projective. (A module is said to be \emph{pure-projective} if it is projective with respect to pure-exact sequences, which is equivalent to being a direct summand of a direct sum of finitely presented modules.) More than four decades ago, Raynaud and Gruson showed:
\begin{displaytext}\label{RG}
 A module is  Mittag-Leffler iff it is countably pure-separable by (in fact, countably generated) pure-projectives.
\end{displaytext}
One may take this as a definition. It implies that a countably generated module is Mittag-Leffler if and only if it is pure-projective. So `Mittag-Leffler' is a generalization of `pure-projective' (and `flat Mittag-Leffler' of `projective').
 The author, more than two decades ago, proved:
 \begin{displaytext}\label{MT}

A module is Mittag-Leffler iff it is positively atomic, 
 \end{displaytext}
  in the sense that the positive primitive (henceforth \emph{pp}) type of every finite tuple in $M$ be finitely generated. This model-theoretic condition---to be explained below---makes many of the algebraic properties of Mittag-Leffler modules very transparent and easy to prove.

  For instance, since the pp type of a tuple  from $M$ does not change when passing to a pure submodule containing it  or when passing to a pure supermodule of $M$, one sees at once:
\begin{displaytext}\label{purechains}
 The union of a pure chain of Mittag-Leffler modules is Mittag-Leffler.
 \end{displaytext}
 Results (\ref{RG}) and (\ref{purechains}) imply:
\begin{displaytext}\label{RGdense}
The countably generated (automatically pure-projective) pure submodules of a Mittag-Leffler module form an $\aleph_1$-dense system,
\end{displaytext} 
in the sense of  \cite[Def.\,2.4]{HT}:
\begin{displaytext}\label{aleph1dense}
An $\aleph_1$-dense system of a module $M$ is an upward directed system $\cal C$ of submodules of $M$ satisfying
\begin{enumerate}  [\upshape (a)]
  \item every countable subset of $M$ is contained in a member of $\cal C$;
  \item $\cal C$ is closed under unions of countable chains.
\end{enumerate}
\end{displaytext}
Note, this definition is a slight weakening of  \cite[Def.IV.1.1]{EM}, where, as an additional condition,  every member of $\cal C$ is required to be countably generated (in which case directedness of $\cal C$ is entailed by (a)).


A Mittag-Leffler module $M$  trivially has $\cal C=\{M\}$ as an $\aleph_1$-dense system of Mittag-Leffler submodules, but the one given in (\ref{RGdense}) is much more meaningful and,
by (\ref{MT}),  the converse of   (\ref{RGdense})  is true as well:
\begin{displaytext}\label{densepure}
A module $M$ is a Mittag-Leffler module iff it has an $\aleph_1$-dense system of Mittag-Leffler modules each of which is \emph{pure} in $M$  iff it has an $\aleph_1$-dense system of (countably generated) pure-projective modules each of which is \emph{pure} in $M$.
\end{displaytext} 
Next we see how to drop `pure' in (\ref{densepure}).  In view of (\ref{MT}), a straightforward argument about finite generation of types in unions of directed systems of submodules yields:
\begin{displaytext}\label{3.6}
Suppose $M$ is the union of a directed system of submodules, $\cal C$, all of whose members are Mittag-Leffler.  Then $M$ is Mittag-Leffler iff every finite tuple $\bar{m}$ of $M$ is contained in a member $C_{\bar{m}}$ of $\cal C$ such that, for every $C\in \cal C$ containing $C_{\bar{m}}$, the pp type of $\bar{m}$ is the same whether considered in  $C_{\bar{m}}$ or in $C$.
  \end{displaytext}
 This condition requires precisely what is needed to make up for the dropped purity condition on the members of $\cal C$. Namely, it says that if a finite system of linear equations, given in matrix form as $\mtx B\bar y=\mtx A\bar m$, has a solution in $M$ (that is, in some  $C\supseteq C_{\bar{m}}$), it already has one in $C_{\bar{m}}$.
(A more general statement for direct limits, but essentially the same simple argument, has been made explicit in  \cite[Lemma 3.6]{PR}, see Lemma \ref{tails} below.)

If such a union $M=\bigcup\cal C$ is not Mittag-Leffler, there must be, by (\ref{3.6}), a tuple $\bar{m}$ in $M$ and an $\omega$-chain $C_0\subset C_1\subset C_2\subset \ldots $ in $\cal C$ such that the pp type of $\bar{m}$ keeps increasing (i.e., more and more equations   $\mtx B\bar y=\mtx A\bar m$   have a solution) when going up the chain. Then, by  (\ref{3.6}) applied to the directed system of the $C_i$ and their union $M'$, the module $M'$ is not Mittag-Leffler, that is:
\begin{displaytext}\label{omega}
If $\cal C$ is a  directed system of submodules of a module, then its union is Mittag-Leffler provided the union of every $\omega$-chain in $\cal C$ is.
  \end{displaytext}
This yields the missing direction of the strengthening of (\ref{densepure}) where purity is no longer required, see \cite[Cor.2.6]{HT} (whose proof is purely algebraic):
\begin{displaytext}\label{fulldense}
A module is Mittag-Leffler iff it has an $\aleph_1$-dense system of Mittag-Leffler modules  iff it has an $\aleph_1$-dense system of (countably generated) pure-projective modules.
\end{displaytext}

\textbf{The flat case.} One can  specialize 
 the above to flat Mittag-Leffler modules by replacing `pure-projective' by `projective' and `finitely presented' by `finitely generated projective' (see  \cite{habil} for this), because 
\emph{flat $+$ pure-projective $=$ projective} and hence \emph{flat $+$ finitely presented $=$ finitely generated projective}. In a `pure situation' this is immediate. For dense systems of not necessarily pure submodules, one may either repeat the above argument under the extra assumption of flatness, or, more elegantly, simply apply the following to the corresponding statements above. 
\begin{displaytext}\label{flatdense}
A module is flat if and only if it has an $\aleph_0$-dense system of flat submodules.
\end{displaytext}
Here, naturally, $\aleph_0$-dense systems are defined like $\aleph_1$-dense systems in (\ref{aleph1dense}) above with `countable' replaced by `finite.' Then condition $(b)$ becomes vacuously true and so an $\aleph_0$-dense system in a module $M$ is simply an upward directed collection $\cal C$ of submodules of $M$ such that every finite subset of $M$ is contained in a member of $\cal C$. (Again, one direction is trivial: take $\cal C=\{M\}$.)

Either way, one deduces at once the following flat versions of (\ref{fulldense}), which was given an algebraic proof in \cite[Thm.2.9]{HT}:
\begin{displaytext}\label{flatMLdense}
A module is flat and Mittag-Leffler iff it has an $\aleph_1$-dense system of flat Mittag-Leffler modules  iff it has an $\aleph_1$-dense system of (countably generated) projective modules.
\end{displaytext}

\textbf{Contents.} In this paper, all of the above, except the previously mentioned flat case, will be done for relativized  Mittag-Leffler modules as introduced and for the first time systematically studied in \cite{habil},  and continued in \cite{AH}, \cite{HT} and others. The concordance for the relativized versions is: 
(\ref{RG}) is Cor.\,\ref{Cor}(1) and Cor.\,\ref{puresep},  
(\ref{MT}) is Thm.\,\ref{MThm},
(\ref{purechains}) is (part of) Cor.\,\ref{Cor}(1),
(\ref{3.6}) is Lemma \ref{tails},
(\ref{omega}) is Cor.\,\ref{closedunion}, and
(\ref{fulldense}) is Cor.\,\ref{denserel} and Prop.\,\ref{ctblesep}.

In Section \ref{KMLSect},  $\cal K$-Mittag-Leffler modules are introduced and their main properties developed. The classical case is that of  $\cal K=\ModR$. This and two others,  $\cal K$, the class $\flat_R$ of flat right $R$-modules, and  $\cal K$,  the class $\sharp_R$ of absolutely pure right $R$-modules, are considered in Section \ref{specialcases}. The fundamental countable separation result (\ref{RG}) (and thus the left-to-right direction of (\ref{fulldense})) above is proved for arbitrary $\cal K$-Mittag-Leffler modules in Section \ref{separation}, see Prop.\,\ref{ctblesep} and Cor.\,\ref{puresep}. For this,  one needs to relativize purity as well, which is done in the first part of Section \ref{Fpure}, the rest of which is preparation for the final
Section \ref{countablygenerated}, where countably generated $\cal K$-Mittag-Leffler modules are characterized.

\textbf{A general remark.} The point of the model-theoretic perspective taken here and in the previous work \cite{habil} is not to make possible all sorts of relativizations---it has been subsequently shown, for instance in \cite{AH}, that this can be done completely algebraically. The point rather is, and this applies to Mittag-Leffler modules in general---relativized or not---that this perspective reveals some features of the concept that have gone unnoticed previously. Take, for instance, the ease with which direct limits of  Mittag-Leffler modules can be treated, as in (\ref{omega}) above (and Section  \ref{imm} below for relativizations),  or the role of definable subcategories as revealed by  the relativized version of  (\ref{MT}), Theorem \ref{MThm}  below (and Sections  \ref{consequences} and \ref{test}), which showed---if stated in different terms---that the very concept of $\cal K$-Mittag-Leffler module depended only on the definable subcategory $\cal K$ generates.

\textbf{Thanks!} I wish to thank the referee for helpful comments resulting in improved  readability of this work.

\section{Preliminaries}
Most of the preliminaries can be found in Mike Prest's two monographs, \cite{P} and \cite{P2}, where the latter is closer in style to current usage. I will largely use its notation---with a few exceptions that are rather self-explanatory. \cite{habil} may also serve as an introduction with all necessary detail.

I freely use the fundamental tool of \emph{pp formulas} (called  \emph{pp condition} in  \cite{P2})\footnote{One may work with Zimmermann's (equivalent) concept of \emph{finite matrix group} and the correponding \emph{p-functors} or the subgroups of finite definition of Gruson and Jensen \cite{GJ} instead.}
and the lattice order $\leq$ between them, often adorned by a class $\cal L\subseteq\RMod$.

\subsection{Modules and tuples}  Module means \emph{left module} (unless otherwise indicated) over an associative ring $R$ with $1$. 

A \emph{tuple} is a finite sequence. Given an $i$-tuple $\bar r = (r_0, \ldots, r_{i-1})$ of ring elements and an $i$-tuple $\bar m = (m_0, \ldots, m_{i-1})$ of elements of a module $M$, 
I use $\bar r \,\bar m$ to denote the corresponding linear combination, i.e.,  $\bar r \,\bar m=\sum_{j<i} r_i m_i$. In doing so I  assume that those tuples match and that the tuple $\br m$ is tacitly transposed to a column vector. (This is for left modules. For right modules it would, naturally, be the other way around.)

It is convenient to expand (this is the syntactic term) a module $M$ by a tuple $\bar m$ from $M$. The resulting object is denoted $(M, \br m)$, an \emph{$i$-pointed module}. Morphisms in the pointed category are supposed to send each entry of a (source) tuple to the corresponding  entry of the (target) tuple.


I use $\langle\, \rangle$ for generation of various objects.  If $C\subseteq M$, then , or more precisely, $\langle\, C\rangle_M$, denotes the submodule of $M$ generated by $C$. In the lattices of pp formulas (see below), 
specifically, for finite generation of types, these symbols are used thus: $\langle \phi\rangle_{\cal L} = \{\psi\,|\, \phi\leq_{\cal L}\psi\}$. For a class $\cal K$ of modules, $\langle\, \cal K\rangle$ denotes the definable subcategory generated by $\cal K$, see Section \ref{defcat} below.

$\Add({\cal K})$    (resp., $\add({\cal K}))$ denotes the smallest class of modules containing $\cal K$  that is closed under  direct summands and (resp., finite) direct sum. 
$\Zg_R$ stands for the Ziegler spectrum of $\RMod$, that is the space of  indecomposable pure-injective left $R$-modules as introduced in \cite{Zie}, cf.\, \cite{P} or \cite{P2}.

\subsection{Formulas and types} The atomic formulas (in the common first-order language of $R$-modules where the scalars figure as unary function symbols) are precisely the linear equations over $R$.  But for us the more important syntactic object are the finite conjunctions of atomic formulas, which are also known as \emph{quantifier-free pp formulas},  the term I use. These are the finite (homogeneous) systems of  linear equations (over the ring). We write them in matrix form with two kinds of variables, $\mtx A \bar x \doteq\mtx B \bar y$, where $\mtx A$ and $\mtx  B$ are matching matrices  over $R$. Given such a  quantifier-free pp formula $\alpha=\alpha(\bar x, \bar y)$, we obtain a typical \emph{pp formula} by existentially quantifying out $\br y$, that is,  $\exists \bar y\alpha$, or,  more explicitly,  $ \exists \bar y\alpha(\bar x, \bar y)$, which we also write as  $\mtx B | \mtx A \bar x$.  

Denote this formula by $\phi$ or $\phi(\bar x)$.  In any given module $M$ it defines the set, $\phi(M)$, of all $l(\bar x)$-tuples $\br a$ for which the inhomogeneous system  $\alpha(\bar a, \bar y)$ has a solution  in $M$. Obviously, $\phi(M)$ forms an additive group, in fact, a subgroup of $M^{l(\bar x)}$.  It is  called a(n \emph{$l(\bar x)$-place}) \emph{pp subgroup} of $M$. Clearly, $\phi(M)$ is the projection of $\alpha(M)$ onto the first $l(\bar x)$ coordinates.

It is easily verified that the $n$-place pp subgroups of $M$   form a sublattice, denoted by $\Lambda^n(M)$, of the subgroup lattice of $M^n$. 

In fact, the $n$-place pp formulas themselves form a lattice, whose meet is conjunction $\wedge$ and whose join is $+$. We  denote this lattice  by $_R\Lambda^n$ (and for right $R$-modules by $\Lambda^n_R$). It turns into $\Lambda^n(M)$ when specialized at $M$. When we specialize it at a class $\cal L$, we obtain a lattice whose order plays a major role in this work and is discussed in the next subsection. 

A \emph{pp $n$-type} is a subset of $_R\Lambda^n$. Of particular importance are those that form a filter in $_R\Lambda^n$, as  the \emph{pp type} of a  given  $n$-tuple $\br a$ in $M$, denoted  $\pp_M(\bar a)$, which is, by definition, the collection of all $n$-place pp formulas $\br a$ satisfies in $M$. The \emph{quantifier-free pp type} of  $\br a$ in $M$, denoted $\qf_M(\bar a)$, is the subset  consisting of all quantifier-free pp formulas in  $\pp_M(\bar a)$.

It is easily seen that every pp formula (and thus every pp type) constitutes a subfunctor of the forgetful functor $\RMod \to \Ab$, from left $R$-modules to abelian groups.  

While often times  the actual lengths  of tuples are irrelevant,  care has to be taken, however, that they match. It is assumed that they (as well as sizes of corresponding matrices) \emph{always} do.

\subsection{The lattice order}\label{latord}
Given pp formulas  $\phi$ and $\psi$, I use $\phi\leq_{\cal L}\psi$ to mean that the sentence $\forall \bar x (\phi(\bar x)\rightarrow\psi(\bar x))$ is true in all members of $\cal L$. We say $\psi$ is \emph{$\cal L$-implied} by $\phi$. When the class $\cal L$ is a singleton $L$, I simple write $\phi\leq_{L}\psi$ for $\phi\leq_{\{L\}}\psi$. When $\cal L=\RMod$, the subscript will be dropped.

Equivalence is denoted by $\sim$  (which stands for `$\leq$ \emph{and} $\geq$'),  adorned or unadorned. So to be  \emph{$\cal L$-equivalent}, i.e., $\phi\sim_{\cal L}\psi$, means that  $\phi$ and $\psi$ are equivalent everywhere in ${\cal L}$ (hence everywhere in $\langle \cal L\rangle$, see Section \ref{defcat}).

All this extends to types in an obvious fashion. Given two pp-$n$-types $p$ and $q$, I say $q$ is \emph{$\cal L$-implied} by $p$ and write $p\leq_{\cal L} q$ if, in every member $L$ of $\cal L$, every realization (= solution) of $p$ realizes $q$ as well. (That is, $p\leq_{\cal L} q$ really means $\bigwedge p  \leq_{\cal L} \bigwedge q$.) Note, if $p=\pp_L(\bar a)$ with $L\in \cal L$, then $p\leq_{\cal L} q$ if and only if $q\subseteq p$ (iff $p\leq  q$), but if $p=\pp_M(\bar a)$ with $M$ \emph{not} in $\cal L$, no such relationship is entailed in general.
The proper context for this relation between types is therefore that of the definable subcategory generated by $\cal L$. 

\subsection{Definable subcategories}\label{defcat}
A \emph{definable subcategory} of left $R$-modules is a full subcategory of $\RMod$ whose object class is the model class of a set of pp implications, i.e., of statements of the form $\forall \bar x (\phi(\bar x)\rightarrow \psi(\bar x))$ where $\phi$ and $\psi$ are pp formulas whose free variables are among those of $\bar x$. This is equivalent to  being closed under direct product, direct limit and pure submodule, cf.\  \cite{purity} or \cite[Thm.3.4.7]{P2}. 

It is clear from the definition that every class  $\cal L$ of left $R$-modules is contained in a smallest definable subcategory (namely the intersection of all of those that contain $\cal L$), which is denoted by $\langle \cal L \rangle$ and called the  \emph{definable subcategory generated by $\cal L$}. Here is a concrete way to obtain it: if ${\rm T}_{\cal L}$ denotes the set of axioms $\{\forall \bar x (\phi(\bar x)\rightarrow \psi(\bar x))\,|\, \phi\leq_{\cal L}\psi\}$ (together with the axioms for left $R$-modules), $\langle \cal L \rangle$ is the class of models of ${\rm T}_{\cal L}$, i.e., the class of modules for which all of those implications hold. That is, $C\in\langle\cal L\rangle$ iff $\phi\leq_{\cal L}\psi$ implies $\phi\leq_{C}\psi$.

Note, two classes $\cal L$ and $\cal L'$ generate the same definable subcategory if and only if the corresponding lattice orderings on pp formulas, $\leq_{\cal L}$ and $\leq_{\cal L'}$, are the same. In particular, $\leq_{\cal L}$ and $\leq_{\langle\cal L\rangle}$ are the same---at least on formulas. 

As a  special case  we have that $\leq_{\RMod}$, $\leq_{\Rmod}$ and $\leq_{_R\Zg}$ (with $_R\Zg$ the left Ziegler spectrum of $R$) are the same. 
More cases are listed in Corollary \ref{sameside} below.

 The lattice orderings of $\cal L$-implication between types as introduced above behave better when the context indicated by the subscript is a definable subcategory (as then compactness of first order logic is available).

\begin{rem}\label{Limpltypes}
Let $p$ and $q$ be pp types of the same arity with $p$ closed under finite conjunction. Then  $p\leq_{\langle\cal L\rangle}q$ if and only if 
for all $\psi\in q$ there is $\phi\in p$ such that $\phi\leq_{\cal L}\psi$.

First of all, $p \leq_{\cal L} q$ if and only if $p \leq_{\cal L} \psi$ for all $\psi\in q$ (and this is true for \emph{any} $\cal L$). But as $\langle\cal L\rangle$ is first order axiomatizable
, the statement $p \leq_{\langle\cal L\rangle} \psi$ is equivalent to the inconsistency  of $p\cup \{\neg\psi\}$ (with  ${\rm T}_{\cal L}$). By compactness, we can replace $p$ by a finite subset. Finite conjunctions of pp formulas being pp, a single $\phi$ from $p$ will do. But inconsistency of $\{\phi\wedge\neg\psi\}$ with  ${\rm T}_{\cal L}$ means 
$\phi\leq_{\cal L}\psi$.
\end{rem}

\begin{rem} This may be a good opportunity to emphasize why sometimes $\langle\cal L\rangle$ and sometimes only $\cal L$. 
In order  to apply compactness to a statement $p \leq_{\cal X} \phi$ when $p$ is an infinite type, we need $\cal X$ to be axiomatizable (so we take $\cal X=\langle\cal L\rangle$). In contrast, for $\phi  \leq_{\cal X}  q$  it doesn't matter, because all that says is that $\phi  \leq_{\cal X} \psi$ for every $\psi\in q$, so we can go interchangeably  with  $\cal X=\langle\cal L\rangle$ or  $\cal X=\cal L$.
\end{rem}

\begin{rem}\textbf{}
\begin{enumerate}[\upshape(1)]
\item 
\cite[Fact 7.7.]{purity} about pure embeddability of any structure from $\langle \cal K\rangle$ in a reduced product of structures from $\cal K$ shows that in order to generate the definable subcategory, one may first take all direct products, then all direct limits, and finally all pure substructures; see also the note before \cite[Prop.3.4.9]{P2}. This is true for any first-order structures, not just modules, cf.\,\cite{purity}.
 \item It is easy to check that for $\cal K\subseteq\cal K'\subseteq\Add(\cal K)$, the preorders $\leq_\cal K$ and $\leq_{\cal K'}$ are the same, hence  $\cal K$,  $\cal K'$, and $\Add(\cal K)$ generate the same definable subcategories.
\end{enumerate}
\end{rem}

\subsection{Elementary duality} 
\emph{Elementary duality} as introduced by Prest and further developed by Herzog is a major tool to switch sides. Namely, for all $n$, this duality, denoted by $\D$,  constitutes an anti-isomorphism between the  lattices $_R\Lambda^n$ and   $\Lambda^n_R$, see \cite{P}, \cite{P2},  \cite{Her},  \cite{ZZ-H}, or \cite{habil} for an introduction.

Elementary duality turns the relation $\leq_{\cal K}$ upside down in $\D\cal K$, the class of all character duals of members of $\cal K$ (here we use character modules obtained by hom-ing into $\mathbb Q/\mathbb Z$). More precisely,  $\phi\leq_{\cal K}\psi$ if and only if $\D\psi\leq_{\D\cal K}\D\phi$.

The use of $D$ is illustrated by the following elegant criterion for a tensor to be zero. Note that we take the same liberty as in linear combinations and write $\bar n \otimes \bar m$ instead of  $\sum_{j<i} n_i  \otimes m_i$ (assuming, as always, that the tuples match).

\begin{fact}[Herzog's Criterion]\label{HC} 
Given  a $k$-tuple  $\bar n$ in a right $R$-module $N$ and  a $k$-tuple $\bar m$ (of the same length) in a  left $R$-module $M$, the following are equivalent.
\begin{enumerate}[\upshape(i)]
\item $\bar n \otimes \bar m=0$
 \item There is a (left) $k$-place pp formula $\phi$ such that $\bar n\in\D\phi(N)$ and  $\bar m\in\phi(M)$.
 \item There is a (right) $k$-place pp formula $\phi$ such that $\bar n\in\phi(N)$ and  $\bar m\in\D\phi(M)$.
\end{enumerate}
\end{fact}

\begin{definition} By  \emph{definably dual classes}  we mean classes  ${\cal K}\subseteq\ModR$ and ${\cal L}\subseteq\RMod$ such that $\D\langle\cal  K  \rangle= \langle\cal  L  \rangle$.
\end{definition}

Since elementary duality goes both ways and is an involution,  this definition is completely symmetric. That is, $\K$ and $\cal L$ are definably dual iff $\langle\cal  K  \rangle=\D\langle\cal  L  \rangle$.

\begin{cor}
Two families, ${\cal K}$ and ${\cal L}$,  of right, resp.\,left  $R$-modules are definably dual, provided any of the following holds.

\begin{enumerate}[\upshape(a)]
\item 
$\phi\leq_\cal L\psi$ iff $\D\psi\leq_\cal K\D\phi$, for all left pp formulas $\phi$ and $\psi$. 
 \item ${\cal L}$ consists of all character duals of members of ${\cal K}$.
 \item Any of the corresponding symmetric statements.
\end{enumerate}
\end{cor}

\begin{conv}\label{conv}
Unless indicated otherwise, $\K$ and $\cal L$ are \emph{\bf definably dual} classes of $R$-modules in the sense of the above definition.
\end{conv}

Prominent examples of this situation are the flat right $R$-modules $\cal K={\flat_R}$ and the absolutely pure left $R$-modules $\cal L={_R\sharp}$  \cite[Prop.3.4.26]{P2}, and the symmetric situation $\cal K=\sharp_R$ and $\cal L={_R\flat}$, considered in Sections  \ref{flat-ML} and \ref{sharp-ML} respectively. See \cite{Tf} for other examples.

\subsection{Finite generation of types}
Following \cite{P}, a pp type $p$ is called \emph{finitely generated} if there is a pp formula $\phi$ with $\langle \phi\rangle=p$. I emphasize that this means that $\phi\leq p$ \emph{and} $\phi\in p$. In \cite[Sect.2.1]{habil} this was relativized to any context $\cal L\subseteq\RMod$. Since there is a subtlety with partial types here (even when the context $\cal L$ is all of $\RMod$), let me be very precise. Consider $p$, the pp type of  a tuple $\br a$ in a module $M$, i.e., $p=\pp_M(\br a)$. A subtype $q\subseteq p$ is said to be \emph{$\cal L$-finitely generated} if $q$ is  $\cal L$-implied by some $\phi\in p$, i.e., when there is $\phi\in p$ such that $q\subseteq\langle \phi\rangle_{\cal L}$, where $\langle \phi\rangle_{\cal L}=\{\psi\,|\,\phi \leq_{\cal L}\psi\}$.

Note, $p$ is $\cal L$-finitely generated iff it is $\cal L$-implied by some formula \emph{it contains}. No such assumption is made if $q$ is partial. The `generator' is only required to be in $p$, not in $q$. This will be of importance below where we will encounter finitely generated quantifier-free types, i.e., the case where $q=\qf_M(\br a)$. According to the definition just made, such $q$ is said to be $\cal L$-finitely generated when there is a pp formula true of $\br a$---quantifier-free or not!---that  $\cal L$-implies $q$. Certainly $q$ is so generated when $p$ is. But we will see cases where $q$ is finitely generated, while $p$ is not.

Consider types $p=\pp_M(\br a)$ and $q=\pp_N(\br b)$ with  $p\leq_{\cal L} q$. If $p$ is $\cal L$-implied by a single formula $\phi$, so is $q$. But $\phi$ may not be in $q$ and thus  $q$ may not be $\cal L$-finitely generated in the sense of the definition above. 

However, if the types in question are  $\cal L$-equivalent, $p\sim_{\cal L} q$, \emph{and}  $\cal L=\langle\cal L\rangle$,  then, by compactness,  every formula in $p$ is  $\cal L$-implied by some formula in $q$ (see Rem.\,\ref{Limpltypes} above). Hence, if $p$ is $\cal L$--implied by  $\phi\in p$, this formula must be  $\cal L$-implied by a formula $\psi\in q$, which then $\cal L$-implies all of $q$, making $q$  in turn $\cal L$-finitely generated as desired. Let us summarize.

\begin{rem}
 Let $p=\pp_M(\br a)$ and $q=\pp_N(\br b)$ be $\langle\cal L\rangle$-equivalent. If $p$ is $\cal L$-finitely generated, then so is $q$.
\end{rem}

I close this section with the simple fact  (known at least  since \cite[Rem.1.6]{habil}) that to check $\cal L$-finite generation of types in a finitely generated module it suffices to consider a generating tuple (cf.\ Lemma \ref{Fpurefg} for a similar phenomenon).

\begin{lem}\label{Latfg}
If the pp type of some generating tuple of a finitely generated module is $\cal L$-finitely generated, then the pp type of every other tuple in that module is $\cal L$-finitely generated as well.
\end{lem}
\begin{proof}
 Suppose  $A$  is generated by $\br a$  and $\pp_A(\br a)$ is $\cal L$-generated by $\phi\in\pp_A(\br a)$. Any other tuple $\br b$ in $A$ can be written as $\mtx B\br a$ for some matrix $\mtx B$ over the ring. Let $\phi_{\mtx B}$ be the pp formula $\exists\br x(\br z=\mtx B\br x \wedge \phi(\br x))$ (in the free variables $\br z$). Clearly, this formula is in $\pp_A(\br b)$, and I claim it $\cal L$-generates this type. Indeed, for any other $\psi=\psi(\br z)\in\pp_A(\br b)$, the formula  $\psi(\mtx B\br x)$  is in  $\pp_A(\br a)$, hence $\cal L$-implied by $\phi$.  It is, finally, straightforward to verify that $\phi\leq_{\cal L} \psi(\mtx B\br x)$ entails  $\phi_{\mtx B}\leq_{\cal L} \psi$.
\end{proof}

\section{Mittag-Leffler modules}\label{KMLSect}
Here we state the main theorem of \cite{habil}---and draw some more or less immediate consequences---about relative Mittag-Leffler modules. It is the basis for everything in this and other work of mine on Mittag-Leffler modules, yet it has, though widely circulated, never been published other than in the thesis \cite[Thm.2.2]{habil}.
(It appeared in the classical case (with no relativization) as the main theorem of \cite{ML}, cf.\,\cite[Thm.1.3.22]{P2}.)
 
\begin{definition}\label{MLdef}\cite[Ch.\,2]{habil}
\begin{enumerate}[\upshape(1)]
 \item Given a family ${\cal K}$ of right $R$-modules, a \emph{${\cal
K}$-Mittag-Leffler module} 
is
a left $R$-module $M$ such that the canonical map
\[\tau^I: (\prod_I N_i) \otimes M\to\prod_I (N_i\otimes M)\]
is monic for all subfamilies $\{N_i : i\in I\}$ of ${\cal K}$. 
\item A  right $R$-module $N$ is called a \emph{test module} for ${\cal
K}$-Mittag-Leffler if, in the previous definition, it suffices to let all the $N_i$ be equal to $N$. We then write $\tau^I_N$ for the transformation $N^I\otimes - \to (N\otimes -)^I$.
\item Given a test module $N$ as before, a \emph{test transformation}  for ${\cal
K}$-Mittag-Leffler is a transformation $\tau^I_N$ as before for a particular set $I$ that suffices to test the ${\cal
K}$-Mittag-Leffler condition above.
\item We omit `$\cal K$-' if ${\cal K}=$ Mod-$R$ and write `$N$-' if
${\cal K}=\{N\}$\,. 
\end{enumerate}
\end{definition}

The two cases of $\cal K$ that led to this definition in \cite{habil} were $\RMod$,   the case of (classical) Mittag-Leffler modules as
introduced in
\cite{RG},  and $\{R_R\}$, which was investigated in \cite{Goo}.

\begin{rem}\label{product}
  Clearly, a direct sum of modules is $\cal K$-Mittag Leffler if and only if every summand is.
Since the canonical map
$(\prod_I\prod_J N_{ij}) \otimes M\to\prod_I\prod_J
(N_{ij}\otimes M)$
factors through the canonical map
$(\prod_I\prod_J N_{ij}) \otimes M\to\prod_I(\prod_J
N_{ij}\otimes M)$, we see that ${\cal K}$-Mittag-Leffler implies ${\cal K}'$-Mittag-Leffler
whenever ${\cal K}'$ consists of products of members of $\cal K$.
\end{rem}

\subsection{Atomic modules}\label{at}
The main theorem below describes $\cal K$-Mittag-Leffler modules in terms of the following concept, which is a variant of a classical notion from first-order model theory.

\begin{definition}\label{D-at}
Suppose $\cal L$ is  a family of left $R$-modules.  An \emph{$\cal L$-atomic}\footnote{Be aware of different usage of this term in \cite[\S3.3]{habil}.}
 module is a left $R$-module $M$ such that the pp type of every tuple in $M$ is $\cal L$-finitely generated.
\end{definition}

In case $\cal L$ is the class of all $R$-modules (or all finitely presented ones), this is what was called \emph{positively atomic} in \cite{habil} and \cite{ML}.

The definition can be reformulated in terms of linear algebra as follows. For every tuple $\bar  a$ in $M$ there is a tuple $\bar b$ in $M$ and a finite system of linear equations, $\alpha(\bar x, \bar y)$ or $\mtx A\bar x \doteq \mtx B\bar y$, with solution $(\bar a, \bar b)$ 
and such that
\begin{displaytext}\nonumber
 if  $\mtx C\bar x \doteq \mtx D \bar z$ is a finite system of linear equations whose 
  inhomogeneous instance  $\mtx C\bar a \doteq \mtx D \bar z$ has a solution in $M$, then for every $L\in\cal L$ and every matching tuple $\bar c$ in $L$, if the  inhomogeneous system 
$\mtx A\bar c \doteq \mtx B\bar y$ has a solution in $L$, so does the inhomogeneous system  $\mtx C\bar c \doteq \mtx D \bar z$.
\end{displaytext}


 \begin{rem}\label{LLatfg} Lemma \ref{Latfg} above says that 
a finitely generated module is $\cal L$-atomic if and only if the pp type of some \emph{generating} tuple is $\cal L$-finitely generated.

Clearly, $\cal L$-atomicity   depends only on $\leq_{\cal L}$, so if this relation is the same as $\leq_{\cal L'}$, then ${\cal L}$-atomic is the same as ${\cal L'}$-atomic. 
(This shows that the concept can be extended to theories, as was done in \cite{habil}).
\end{rem}

\subsection{The main theorem}

\begin{thm}\label{MThm}{\rm \cite[Thm.2.2]{habil}}
Let ${\cal K}$ and ${\cal L}$  be definably dual classes  of right and left  $R$-modules (i.e., such that 
$\phi\leq_\cal L\psi$ iff $\D\psi\leq_\cal K\D\phi$, for all left pp formulas $\phi$ and $\psi$). 

 Then a  module is ${\cal K}$-Mittag-Leffler if and only if it is ${\cal L}$-atomic.
\end{thm}
\begin{proof}
 The case $\cal K=\RMod$ is  contained in \cite[Thm.2.2]{ML} and in \cite[Thm.1.3.22]{P2}. The proof of the general case is almost verbatim the same (and carried out in detail in \cite{habil}). The main adjustment to be made is to replace every unrelativized occurrence of an implication $\leq$ between left pp formulas by the relativized implication $\leq_{\cal L}$ and every implication between right pp formulas by the relativized implication $\leq_{\cal K}$. 
\end{proof}

Because of space concerns I state the following more or less  immediate consequences  without proof. 

\begin{cor}\label{Cor}\cite[Cor.2.4]{habil}
\begin{enumerate}[\upshape(1)]
\item  A module $M$ is ${\cal K}$-Mittag-Leffler iff every finite subset
of $M$ is contained in a pure submodule of $M$ which is ${\cal K}$-Mittag-Leffler. In particular, the class of  $\cal K$-Mittag-Leffler modules is closed under taking  unions of pure chains.
\item \label{closed}The class of $\cal K$-Mittag-Leffler modules is closed under pure
submodules and pure extensions.
\item A direct sum is $\cal K$-Mittag-Leffler iff all summands are $\cal K$-Mittag-Leffler.
\item A finitely generated module is finitely presented iff it is Mittag-Leffler.
\item If $N$ is a finitely generated pure submodule of a ${\cal K}$-Mittag-Leffler module $M$
(hence $N$ is finitely presented), then $M/N$ is ${\cal K}$-Mittag-Leffler too.
\end{enumerate}
\end{cor}

\subsection{Definable subcategories}\label{consequences}
As was stressed in its original form---albeit in different terminology---one of the most powerful direct consequences of this theorem is that the concept of $\cal K$-Mittag-Leffler module is invariant under passing to the generated definable subcategory. Many a preservation result for classes $\cal K$ found in the literature follow directly from this. We make it explicit in (1) below. The other statements are special cases contained in one way or another in \cite[Ch.\,2]{habil}, for (3), see \cite[Rem.2.3(b)]{habil}.

\begin{cor}\label{sameside}
 Let ${\cal K}$ and ${\cal K}'$  be  families of right $R$-modules.  
 
Then a  module is ${\cal K}$-Mittag-Leffler if and only if it is ${\cal K}'$-Mittag-Leffler, provided any of the following holds.

\begin{enumerate}[\upshape(a)]
\item ${\cal K}$ and ${\cal K}'$ generate the same definable subcategory.
 \item $\Add({\cal K})=\Add(\cal K')$  (or $\add({\cal K})=\add(\cal K')$).
 \item ${\cal K}$ and ${\cal K}'$ have the same associated Ziegler-closed set, i.e., $$\Zg_R\cap \langle \cal K\rangle =\Zg_R\cap \langle \cal K'\rangle.$$
 \item ${\cal K}'$ consists of all direct products of members of $\cal K$.
 \item Every member of $\cal K$ is elementarily equivalent to a member of $\cal K'$, and vice versa.\qed
\end{enumerate}
\end{cor}

The converses do not hold. For example, over a left noetherian ring, every left module is $\flat$-Mittag-Leffler, cf.\ Cor.\,\ref{noeth}(2) below. Hence the latter is the same as $\{0\}$-Mittag-Leffler, while  $\langle\flat\rangle$ and $\{0\}$ are two different definable subcategories.

Since the definable subcategory generated by $\Rmod$ is $\RMod$, as a special case of (1) we get that $\Rmod$-Mittag-Leffler modules are Mittag-Leffler  \cite[Lemma 3.1]{Fac}.

\subsection{Direct limits}\label{imm}
Consider  an upward directed poset  $I=(I, \leq)$ and  an  $I$-directed system $\cal M = \{ M_i\}_{I}$ of modules with connecting maps $f_i^j: M_i\to M_j$ for all $i\leq j$ in $I$ (and $f_j^k f_i^j = f_i^k$ for $i\leq j\leq k$). Let  $M$ be its direct limit with canonical maps $f_i: M_i \to M$ (such that $f_j f_i^j = f_i$ for $i\leq j$). 

 A \emph{tail} in $\cal M$ is a set of the form $\{f_i^j(\bar{m}_i) | i\leq j\in I\}$, where $i\in I$ and $\bar{m}_i$ is a tuple in $M_i$. Clearly, the canonical maps $f_j$ send each entry $\bar{m}_j=f_i^j(\bar{m}_i)$ in this tail to the same tuple,  $\bar{m}$ say, in $M$. All entries in this tail being preimages of $\bar{m}$ under the corresponding canonical maps, we also call this tail a \emph{tail of $\bar{m}$}. Clearly, any two, and hence any finitely many tails of a given tuple merge at some $j$th entry. So, as far as eventual behavior in tails is concerned, it makes no difference with  which tail we start. Let's make this explicit.
 
\begin{rem}\label{limittruth}
 A tuple $\bar{m}$ in the limit $M$ satisfies a pp formula $\phi$ if and only if a(ny)  tail  of $\bar{m}$ eventually does.
\end{rem}
 
 Suppose now the $M_i$ are  $\cal K$-Mittag-Leffler, hence $\cal L$-atomic. We want to know what makes the limit $\cal K$-Mittag-Leffler, i.e., $\cal L$-atomic, too. For this, as we will see,
 it suffices to consider countable chains in the direct system at hand. 
 
By atomicity,   the pp type of the $j$th entry $\bar{m}_j$ in a tail $\{\bar{m}_j | j\geq i\}$ is $\cal L$-generated by some pp formula $\phi_j$. For every  $k\geq j$, the $k$th entry is an image of the $j$th and thus satisfies $\phi_j$. Therefore these formulas form a descending chain, $\phi_i \geq_{\cal L} \phi_j \geq_{\cal L} \phi_k$ for all $i\leq j \leq k$ in $I$. And they are all contained in the pp type of the limit $\bar{m}=f_i(\bar{m}_i)$.
 
 We say a tail as above $\cal L$-\emph{stabilizes at $\phi_j$} if $\phi_k \sim_{\cal L}  \phi_j$ for all $k\geq j$.
 

\begin{lem}{\rm(Essentially \cite[Lemma 3.6]{PR})}\label{tails} Given  a limit tuple $\bar{m}$ and a pp formula $\phi$, the following are equivalent.
 \begin{enumerate}  [\upshape (i)]
 \item $\pp_M(\bar{m})$ is $\cal L$-generated by  $\phi$.
 \item A(ny) tail $\{f_i^j(\bar{m}_i) | i\leq j\in I\}$ of $\bar{m}$ $\cal L$-stabilizes at some $\phi_j$ (in which case $\phi_j\sim_{\cal L}  \phi$  and hence $\phi_k \sim_{\cal L}  \phi$ for \emph{all} $k\geq j$).
 \end{enumerate}
\end{lem}
\begin{proof} Let $\{\bar{m}_j | j\geq i\}$ be a tail of the tuple $\bar{m}$. 

 (i) $\Rightarrow$ (ii): 
For $\bar{m}$ to satisfy  $\phi$ in $M$, the tail must eventually satisfy it in the respective $M_j$'s, Rem.\,\ref{limittruth}. As $\phi_k$ ${\cal L}$-generates $\pp_{M_k}(\bar{m}_k)$, this implies $\phi_k\leq_{\cal L} \phi$ (from some $j$ on). But all these formulas are contained in $\pp_M(\bar{m})$ and are therefore, conversely,   ${\cal L}$-implied by $\phi$, so the tail stabilizes at $\phi_j$  and $\phi_k \sim_{\cal L}  \phi$ from $j$ on.
 
(ii) $\Rightarrow$ (i). Suppose the tail stabilizes at $\phi_j$. To verify that $\pp_M(\bar{m})$ is $\cal L$-generated by $\phi_j$, let $\psi\in\pp_M(\bar{m})$. Then the tail eventually satisfies  $\psi$. In particular, there is $M_k$ with
$k\geq j$ where $\bar{m}_k$ satisfies $\psi$.
Then $\phi_k\leq_{\cal L}  \psi$ and so $\phi_j\sim_{\cal L}  \phi_k$ $\cal L$-implies $\psi$ as well. 
 \end{proof}

%
\begin{definition}\label{wlimit}\textbf{}
\begin{enumerate}[\upshape(1)]
\item By an \emph{$\omega$-limit} we mean a direct limit of a chain of order type $\omega$.
 \item
 By an \emph{$\omega$-sublimit} of the system $\cal M=\{ M_i, f_i^j \}_{i\leq j\in I}$  we mean an $\omega$-limit of the form
 $M_{i_0}\stackrel{f_{i_0}^{i_1}}\to M_{i_1} \stackrel{f_{i_1}^{i_2}}\to M_{i_2} \stackrel{f_{i_2}^{i_3}}\to \ldots \stackrel{f_{i_{k-1}}^{i_k}}\to M_{i_k} \stackrel{f_{i_k}^{i_{k+1}}}\to\ldots$ ($k<\omega$),  where $i_0< i_1< i_2< \ldots < i_k < \ldots $  ($k<\omega$) in $I$. 
 \item Call the directed system $\cal M$ (or its limit) \emph{$\omega$-closed} if every $\omega$-sublimit (as a structure) is itself a member of $\cal M$.
 \end{enumerate}
\end{definition}

\begin{prop} {\rm \cite[Prop.2.2]{HT}}
A direct limit of   $\cal K$-Mittag-Leffler modules is  $\cal K$-Mittag-Leffler provided all $\omega$-sublimits are. Thus, an $\omega$-closed direct limit of   $\cal K$-Mittag-Leffler modules is  $\cal K$-Mittag-Leffler.
\end{prop}
\begin{proof}
 In the notation introduced in the beginning of this subsection, assume the limit $M$ is \emph{not}  $\cal K$-Mittag-Leffler. By Thm.\ \ref{MThm}, there's a tuple  $\bar{m}$ in $M$ whose pp type is not 
$\cal L$-finitely generated. Then, by the previous lemma, some (hence every) tail $\{f_i^j(\bar{m}_i) | i\leq j\in I\}$ of $\bar{m}$ does not stabilize. Hence there's a sequence  $i\leq i_0< i_1< i_2< \ldots < i_k < \ldots $  ($k<\omega$) in $I$ for which the pp formulas corresponding to this subtail form a strictly descending chain:   $\phi_{i_0}>_{\cal L} \phi_ {i_1}>_{\cal L} \phi_{i_2}>_{\cal L}  \ldots >_{\cal L} \phi_ {i_k} >_{\cal L} \ldots $  ($k<\omega$) in $I$. 

Consider the sublimit $N$ of the $M_{i_k}$ (an $\omega$-limit) with canonical maps $g_k: M_{i_k}\to N$ ($k<\omega$) and the tuple $\bar{n}=g_k(f_i^{i_k}(\bar{m}_{i}))$ therein. Its tail $\{f_i^{i_k}(\bar{m}_i) | k<\omega\}$ does not stabilize and so its pp type is, by the previous lemma,  not $\cal L$-finitely generated and $N$, by Thm.\ \ref{MThm},  not $\cal K$-Mittag-Leffler, as desired. 
\end{proof}

We show in Section \ref {ctblechains} below what it takes to make an $\omega$-limit  $\cal K$-Mittag-Leffler.

\begin{cor}\label{closedunion} The union of an  $\omega$-closed direct system of
 of  $\cal K$-Mittag-Leffler submodules is  $\cal K$-Mittag-Leffler.\qed
\end{cor}

\subsection{Density}
In the spirit of \cite{EM}, Herbera and Trlifaj \cite[Def.2.5]{HT} call a collection, $\cal C$,  of submodules of a module $M$ a \emph{$\kappa$-dense system} (in $M$) provided that

\begin{enumerate}[\upshape(a)]
 \item every subset of $M$ of power less than $\kappa$ is contained in some $C\in\cal C$; 
 \item $\cal C$ is  closed under chains of length less than $\kappa$;
 \item $\cal C$ is directed, i.e., every two members of $\cal C$ are contained in a member of $\cal C$.
 \end{enumerate}
 
The case of interest here is $\kappa=\aleph_1$ (but the general definition reveals the role of "$<$" more visibly). 


 \begin{cor} {\rm \cite[Thm.2.6]{HT}}\label{denserel}
 A module is  $\cal K$-Mittag-Leffler if and only if it possesses  an $\aleph_1$-dense system of  $\cal K$-Mittag-Leffler submodules. Furthermore, in that case these submodules can always be chosen to be countably generated. 
\end{cor}
\begin{proof}
A module with an $\aleph_1$-dense system of  $\cal K$-Mittag-Leffler submodules is the union of that system, by (a). This union is  $\omega$-closed by (b), and so the module is  $\cal K$-Mittag-Leffler by the previous corollary.

 That, conversely,  every  $\cal K$-Mittag-Leffler  module has such an $\aleph_1$-dense system 
will be immediate from Prop.\ \ref{ctblesep} below, which shows that one can take for $\cal C$ the set of all countably generated $\cal L$-pure submodules.
\end{proof}

\subsection{Test modules and test maps}\label{test} It follows from the Main Theorem that a test module  for $\cal K$-Mittag-Leffler in the sense of Definition \ref{MLdef} above is a module that generates the same definable subcategory as does $\cal K$. Another way of saying this is  (1) of the following corollary. The rest are immediate consequences of that---for (4) invoke 
the L\"owenheim-Skolem Theorem.

\begin{cor}\cite[Cor.2.6]{habil}\label{testmod}
 Test modules for $\cal K$-Mittag-Leffler are:
\begin{enumerate}[\upshape(1)]
 \item Any right module $N$ such that $\leq_{\cal K}$ and $\leq_N$ are the same.
 \item $N$ the direct sum (or product)  of all pure-injective indecomposables in the Ziegler-closed set corresponding to  $\cal K$.
 \item $N$ the direct sum (or product) of one model of every completion of the first-order (or the implicational) theory of $\cal K$. 
 \item Any module elementarily equivalent to a test module; so there are test modules of power $\card$.\qed
\end{enumerate}
\end{cor}

In the remainder of this subsection I discuss, without detailed proofs, a bound on the \emph{exponent} $I$ occurring in a test transformation $\tau_N^I$, see Definition \ref{MLdef}. To minimize this exponent, we need some
more notation.
Given a lattice $\Lambda$\,, let $\triangle(\Lambda)$
\label{delta}
be the smallest cardinal $\kappa$ such that every ideal  is generated by
$\kappa$ elements. Note that $\triangle(\Lambda)$   is finite if
and only if it is 1 if and only if all ideals  in $\Lambda$ are
principal iff $\Lambda$ has no infinite (strictly) ascending chains, in
which case $\Lambda$ is called a \emph{noetherian lattice}, cf.
\cite[Ch.III, ex.10]{S}. Since finitely generated  left ideals are pp definable in
$R_R$\,, we see that  $\triangle(R_R)=1$ iff $R$ is left noetherian.
For a module $N$\,, let  $\triangle(N)$
denote the
supremum of the cardinals $\triangle(\Lambda^n(N))$, where, recall, $\Lambda^n(N)$ is the lattice of $n$-place pp subgroups of $N$ and $n$ runs over the natural numbers.
Since there are no more than $\card$ pp formulas altogether, $\triangle(N)$ is always bounded by this number. 

 
\begin{fact}\label{testfact}
 {\rm \cite[Cor.2.6]{habil}.} 
\begin{enumerate}[\upshape(1)]
 \item Given a test module $N$ for  $\cal K$-Mittag-Leffler, $\tau^{\triangle(N)+|N|}_N$ is a test transformation.
 \item If this test module $N$ has power $\card$, then $\tau^N_N$ is a test transformation.
\end{enumerate}
\end{fact}

It was noted in 
\cite[Rem.2.3(c),(d)]{habil} that  in case  $\triangle(N)=1$,  the entire exponent in $\tau^{\triangle(N)+|N|}_N$
can be taken to be $1$, and  \emph{every} left module $M$ is
$N$-Mittag-Leffler. 
In particular, over a left noetherian ring $R$ all left modules
are $R$-Mittag-Leffler. This goes back to \cite[p.1]{Goo}; see Cor.\,\ref{noeth}(2) below for the converse.

\section{Special cases}\label{specialcases}
Here we specialize some of the previous results to the following three cases: $\cal K=\ModR$ and   $\cal L=\RMod$,  $\cal K=\flat_R$ and   $\cal L= {_R\sharp}$,  and  $\cal K=\sharp_R$ and   $\cal L= {_R\flat}$ (notation as introduced after CONVENTION \ref{conv}).

\subsection{(Classical) Mittag-Leffler modules, i.e., $\cal K=\ModR$}  Corollary \ref{sameside} yields at once:

\begin{cor}{\rm \cite[Rem.2.3]{habil}}
Given a left $R$-module $M$, the following are equivalent.

\begin{enumerate}[\upshape(i)]
\item $M$ is Mittag-Leffler.
\item $M$ is  ${\cal K}$-Mittag-Leffler and $\langle {\cal K}\rangle=\ModR$. 
\item $M$ is   $\Zg_R$-Mittag-Leffler. 
\item $M$ is  ($\modR$)-Mittag-Leffler.
 \end{enumerate}
\end{cor}

\begin{cor}
Besides the test modules one obtains by specializing Cor.\,\ref{testmod} to the case $\cal K=\ModR$, another test module is $P_*$, the direct sum of one representative of each isomorphism type from $\modR$. (This module has power at most $\card$.)
\end{cor}

\subsection{$\flat$-Mittag-Leffler modules}\label{flat-ML} The first concrete instance of relativized Mittag-Leffler module in the literature seems to have been considered in  \cite{Goo}  (not under this name) with $\cal K = {\flat_R}$. By a result of Zimmermann  \cite{Zim}, $\leq_{\flat_R}$ and $\leq_{R_R}$ are the same (i.e., ${\flat_R}$ and ${R_R}$ generate the same definable subcategory), see also \cite[Thm.2.3.9]{P2}. Thus $\flat$-Mittag-Leffler is the same as $R_R$-Mittag-Leffler, or $R$-Mittag-Leffler for short. The following basic properties of these modules are immediate from the main theorem. All but condition (v) (from \cite[Cor.2.7]{habil}) are contained in \cite{Goo}.

\begin{cor}\label{Goo}\label{noeth}\textbf{}
\begin{enumerate}[\upshape(1)]
\item For any left $R$-module $M$ the following are
equivalent.
\begin{enumerate}  [\upshape (i)]
\item $M$ is $\flat$-Mittag-Leffler.
\item $M$ is $R_R$-Mittag-Leffler, i.e., $R_R$ is a test module.
\item $R^R \otimes M \to M^R$ is monic
(i.\,e. $R^R \otimes - \to (R\otimes -)^R$ \quad is a test
transformation for $\flat$-Mittag-Leffler).
\item  For every
finitely generated  submodule $A$ of $M$, the inclusion $A\subseteq M$ factors through a finitely presented module.
\item All quantifier-free types realized in $M$ are finitely generated.
\end{enumerate}

\item
\cite[p.1]{Goo} Every left $R$-module is
$R$-Mittag-Leffler iff $R$ is left noetherian.
\end{enumerate}
\end{cor}
\begin{proof}
The second part follows from (iv) of the first.

Since $\le_{\flat_R}$ and $\le_{R_R}$ are the same, we see from the theorem that $R_R$ is a test module, hence
we have
(i) $\Leftrightarrow$ (ii). To prove
(ii) $\Leftrightarrow$ (iii) we are thus left with
checking that the exponent can be taken to be $R$\,. If $R$ is infinite,
$\card=|R|=\triangle(R_R)\cdot |R_R|$\,, hence we get the desired
exponent from this and the second part of the above corollary. If $R$ is
finite, $\triangle(R_R)=1$ and we get exponent $1$ from the discussion after Fact \ref{testfact}.

(iv) $\Rightarrow$ (v): Let $A=\langle\br a\rangle_M$ and $(A, \br a)\to (P, \br c)\to (M, \br a)$ the factorization given, with $P$ finitely presented. Then $\qf_M(\br a)=\qf_A(\br a)\subseteq\pp_P(\br c)\subseteq\pp_M(\br a)$, so the pp formula that generates $\pp_P(\br c)$ also generates $\qf_M(\br a)$.

(Of course this suffices for the proof, but it's interesting to note a direct proof of the converse: if $\phi\in\pp_M(\br a)$ generates $\qf_M(\br a)$, then any finitely presented free realization $(P, \br c)$ of $\phi$ maps to $(M, \br a)$; it remains to find $(A, \br a)\to (P, \br c)$, which is easily done by extending the map $\br a\to\br c$ by linearity; this is well defined, for if $\br r\br a=0$, then $\br r\br x\dot=0$ is in $\qf_A(\br a)$, hence implied by $\phi$, and so $\br r\br c=0$.)

(i) $\Leftrightarrow$ (v). By the theorem,
(i) is equivalent to every pp type in $M$ being $\D
(\flat_R)$-finitely generated. We will show that this is the same as every
quantifier-free type in $M$ being ($\RMod$-) finitely generated. First of all, by the duality of
flat and absolutely pure \cite{Her} (for another reference, if somewhat in disguise, see \cite[Prop.3.4.26]{P2}), $\le_{\D
(\flat_R)}$ is the same as $\le_{_R\sharp}$\,. So we have to show
that $\pp_M(\bar{a})$ is $\sharp$-finitely generated iff the type  $\qf_M(\bar{a})$ is
finitely generated. Because of its significance, we make this claim a separate lemma.
\end{proof}

\begin{lem}\label{qfgen} Given a tuple $\br a$ in a module $M$, the following are equivalent.
\begin{enumerate}  [\upshape (i)]
\item $\qf_M(\bar{a})$ is
finitely generated.
\item $\qf_M(\bar{a})$ is
 $\sharp$-finitely generated.
\item $\pp_M(\bar{a})$ is $\sharp$-finitely generated. 
\end{enumerate} 
Moreover, the same pp (not necessarily quantifier-free!)  formula  works in all three cases.
\end{lem}
\begin{proof}
 For the nontrivial implications, first suppose $\psi\in \pp(\bar{a})$ generates $\qf_M(\bar{a})$. To show this formula $\sharp$-generates $\pp_M(\bar{a})$, let $\varphi\in \pp(\bar{a})$ and $N$ be absolutely pure.
We have to show that $\psi(N)\subseteq\varphi(N)$\,. 
Since $\bar{a} \in \varphi(M) \subseteq \ann_M\D\varphi(R_R)$\,, for all
$\bar{r}\in\D\varphi(R_R)$ the formula $\bar{r}\, \bar x \doteq 0$ is in
$\qf_M(\bar{a})$\,, hence $\psi\le\bar{r}\,\bar{x} \doteq 0$. In particular,
$\psi(N)\subseteq\ann_N\D\varphi(R_R)$. But as $N$ is absolutely pure, by \cite[Prop.2.3.3]{P2}, this latter set \emph{is} $\varphi(N)$.

Now assume $\qf_M(\bar{a})$ to be $\sharp$-generated by
$\psi\in\pp_M(\bar{a})$\,. To show $\psi$ generates $\qf_M(\bar{a})$,
  let $\alpha\in\qf_M(\bar{a})$ and $N$ be any module and $\bar b\in \psi(N)$. We have to show $\bar b\in \alpha(N)$. Let $N\subseteq E$ be an injective envelope. Being existential, the formula 
  $\psi$ is satisfied by $\bar b$ also in $E$. Since
$\psi\le_{\sharp}\alpha$, therefore $\bar b\in \alpha(E)$. But $\alpha$ is quantifier-free and $\bar b$ is in $N\subseteq E$, hence  $\bar b\in \alpha(N)$, as desired.
\end{proof}

\subsection{$\sharp$-Mittag-Leffler modules}\label{sharp-ML}
Another instance is $\cal K=\sharp_R$. As $\D(_R\flat)$ and $\sharp_R$ generate the same definable subcategory \cite[Prop.3.4.26]{P2}, we may then take $\cal L={_R\flat}$, and  the theorem shows that a module is $\sharp$-Mittag-Leffler iff it is $\flat$-atomic. Symmetrically,  $\D(\flat_R)$ and $_R\sharp$ generate the same definable subcategory, hence a module is $\flat$-Mittag-Leffler iff it is $\sharp$-atomic, a fact we haven't exploited in the previous subsection.

 The next example, though rather trivial, may be instructive as to what the role of $\cal K$ and $\cal L$.

\begin{exam} \label{torsion}
 Let $R$ be a (not necessarily commutative) domain and let $\Div$ denote the class of divisible right $R$-modules. Then any torsion left $R$-module is $\Div$--Mittag-Leffler (hence also  $\sharp$-Mittag-Leffler, for every absolutely pure module over a domain  is divisible): as a divisible module tensored with a torsion module is zero and direct products of divisible modules are divisible, the canonical map figuring in the definition  of  Mittag-Leffler module is vacuously monic, cf.\ \cite{Tf}.
  \end{exam}

\begin{exam} \label{notfp}
 Let $R$ be a (not necessarily commutative) domain with a two-sided ideal $I$ that is not finitely generated as a left ideal. Then the cyclic left $R$-module $R/I$ is not finitely presented. Yet it  is $\sharp$-Mittag-Leffler (even $\Div$--Mittag-Leffler), for it is torsion, see Example \ref{torsion}.  
 
In order to elucidate the concept, let us give another proof via atomicity. We have to verify that  $R/I$ is $\flat$-atomic. We know from Lemma \ref{Latfg} that it is enough to verify that the generator $1/I$ has $\flat$-finitely generated pp type. Now, this type is equivalent to saying $Ix=0$ or, more precisely, to the set of all formulas $rx=0$ with $r\in I$. Let $s$ be any non-zero element of $I$. I claim that the pp type of $1/I$ in $R/I$ is  $\flat$-generated by the formula $sx=0$. First of all, of course this formula is in the type, for $s(1/I)=0$ in $R/I$. To see, it $\flat$-implies the the rest of it, we have to verify that $sx=0\le_{\flat} rx=0$ for all $r\in I$. It satisfies to verify this implication in the left regular module $R$. But there, the only element satisfying $sx=0$ is $0$. And $0$ realizes \emph{any} pp type. (This shows that, given any pp type $p$ in any left $R$-module $M$, if $p$ contains a formula $sx=0$ for some non-zero $s\in R$, then this formula $\flat$-generates $p$.)
 \end{exam}

For other relativized Mittag-Leffler modules see \cite{Tf}.

\section{Purity}\label{Fpure}
Here we investigate three kinds of relativized purity.  The first of these is needed in the next section, the other two in the last.
 
\subsection{Left purity: pure monomorphisms} 
To prove the relativized version, Corollary \ref{puresep} below, of Raynaud and Gruson's separability result (\ref{RG}) at the top of this paper, one needs relativized versions of purity. An algebraic one was given in \cite[Thm.2.5(ii)]{HT} and a rather model-theoretic one in \cite[Prop.3.3]{Tf}. Here we show that they are the same. As it turns out, this purity goes back to Herzog's thesis
\cite[Sect.8]{Her}, where it was called  \emph{local}.

\begin{definition}[Herzog]\label{FpureHerzog}
An \emph{$\cal L$-pure monomorphism} is a monomorphism $f: N\to M$ (of left $R$-modules) such that $K\otimes f$  is a monomorphism for every $K\in\D\langle\cal L\rangle$ (the dual of the definable category generated by $\cal L$). An  \emph{$\cal L$-pure submodule}  of a module $M$ is  a submodule $N$ whose inclusion in $M$ is $\cal L$-pure, i.e., such that the corresponding map
$K\otimes N \to K\otimes M$ is monic for every  $K\in\D\langle\cal L\rangle$. As usual, we drop curly brackets around singletons.
\end{definition}

\begin{rem}
 It is an easy exercise using Herzog's criterion (and elementary duality) to show that it suffices to check the above monomorphism for all $K\in\cal K$, where $\cal K$ is any class such that $\langle\cal K\rangle=\D\langle\cal L\rangle$. In particular, $\cal L$-purity is the same as $\langle\cal L\rangle$-purity.
\end{rem}

 The next result  is essentially \cite[Prop.8.3]{Her}. Herzog's original proof was in terms of closed sets of the Ziegler spectrum. The above remark allows us to pass to a class that is closed under product and thus to apply McKinsey's Lemma, see \cite[Cor.9.1.7, p.417]{H} (or  \cite[Cor.6.3]{purity}), instead  of Ziegler's indecomposability calculus. I include such a proof.

\begin{lem}[]\label{FpureDef} 

 Given a submodule $N$ of a module $M$, the following are equivalent.
 \begin{enumerate}  [\upshape (i)]
\item $N$ is an $\cal L$-pure submodule of $M$.
\item For every tuple $\bar{n}$ in $N$, its pp type in $N$ is $\langle\cal L\rangle$-equivalent to its  pp type in $M$.
\end{enumerate}
\end{lem}
\begin{proof}
Consider a  tuple $\bar{n}$ in $N$ and denote its pp type in $M$ by $p$ and its pp type in $N$ by $q$. Clearly, $q\subseteq p$, hence, by compactness, these types are  $\langle\cal L\rangle$-equivalent  precisely when every $\phi\in p$ is  $\langle\cal L\rangle$-implied by some $\psi\in q$. As $\leq_\cal L$ is the same as $\leq_{\langle\cal L\rangle}$,  the last two conditions are indeed equivalent. 

It is an(other) easy exercise using Herzog's criterion  to show that any of these two conditions implies the first.
So we are left with the less straightforward implication (i) $\Rightarrow$ (ii). We have to show that for every tuple $\bar{n}$ in $N$ and every pp formula $\phi$ it satisfies in $M$, there is a pp formula $\psi\leq_{\cal L} \phi$ that  $\bar{n}$ satisfies in $N$. To this end, consider $\phi\in p$ and let let $\cal K$ be any class closed under product such that $\langle\cal K\rangle=\D\langle\cal L\rangle$. Notice, this latter equality means that  $\psi\leq_{\cal L} \phi$ iff $\D\phi\leq_{\cal K} \D\psi$. 

First note that for every choice of $\bar{k}\in \D\phi(K)$  in any $K\in \cal K$, there is $\psi\in q$ such that $\bar{k}\in \D\psi(K)$. Indeed, as $\bar{n}$ satisfies $\phi$, we have $\bar{k}\otimes \bar{n}=0$ in $K\otimes M$, hence, by (i), also in  $K\otimes N$, which yields $\psi\in q$ such that $\bar{k}$ satisfies $\D\psi$ in $K$.

This shows that $\forall\bar{x}(\D\phi\rightarrow \bigvee_{\psi\in q}\D\psi)$ holds in every member of $\cal K$. McKinsey's Lemma yields a single $\psi\in q$ such that $\forall\bar{x}(\D\phi\rightarrow \D\psi)$ holds in every member of $\cal K$, i.e., $\D\phi\leq_{\cal K} \D\psi$. By duality, $\psi\leq_{\cal L} \phi$, as desired.
\end{proof}

We see immediately how this looks when all the types $p$ are $\cal L$-finitely generated. 

\begin{cor}\label{FpureDefCor} 
Suppose $N$  is a submodule of  an $\cal L$-atomic module $M$. Given  a tuple $\bar{n}$  in $N$, let  $\phi_{\bar{n}}$ be an $\cal L$-generator of $\pp_M(\bar{n})$.

 Then the following are equivalent.
 \begin{enumerate}  [\upshape (i)]
\item $N$ is an $\cal L$-pure submodule of $M$.
\item For every tuple $\bar{n}$ in $N$, the type $\pp_N(\bar{n})$ is $\langle\cal L\rangle$-equivalent to  $\phi_{\bar{n}}$.
\item For every tuple $\bar{n}$ in $N$,  the type $\pp_N(\bar{n})$ contains a pp formula $\psi_{\bar{n}}\leq_{\cal L} \phi_{\bar{n}}$ (and hence $\psi_{\bar{n}}\sim_{\cal L} \phi_{\bar{n}}$).
\end{enumerate}

\noindent
In this case $N$ is  $\cal L$-atomic too.
\end{cor}
\begin{proof}
Just note that $\phi_{\bar{n}}$ is $\cal L$- (hence $\langle\cal L\rangle$-) equivalent  to $\pp_M(\bar{n})$.
\end{proof}

To check $\cal L$-purity of a  finitely generated submodule it suffices to consider a generating tuple in Lemma \ref{FpureDef}. The argument is standard and similar to the one in  Lemma \ref{Latfg}:

\begin{lem}\label{Fpurefg}
Let $\br a$ be a tuple in $M$ and $N$ the submodule it generates. Then the following are equivalent.
 \begin{enumerate}  [\upshape (i)]
\item $N$ is an $\cal L$-pure submodule of $M$.
\item $\pp_N(\br{a})$ is $\langle\cal L\rangle$-equivalent to $\pp_M(\br{a})$.
\item For every pp formula $\phi$ in $\pp_M(\br{a})$, there is a pp formula $\psi\leq_{\cal L} \phi$  (and hence $\psi\sim_{\cal L} \phi$) in $\pp_N(\br{a})$.
\end{enumerate}
\end{lem}
\begin{proof}
As $\br a$ generates $N$, for every tuple $\br n$ in $N$ there is a matrix $\mtx B$ such that $\br n = \mtx B\br a$. For every pp formula $\phi$ in the type of $\br n$ (in $N$ or $M$), the pp formula $\phi(\mtx B\br x)$  is in the corresponding type of $\br a$. Using this it is straightforward to verify that (ii) entails  $\langle\cal L\rangle$-equivalence of $\pp_N(\br{n})$ and $\pp_M(\br{n})$.
\end{proof}

The next, trivial example may serve to clarify some of the purity issues involved in relative implication (localized at a definable subcategory).

\begin{exam} Consider $\cal K$ the class of divisible (=injective) and $\cal L$ the class of torsionfree (=flat) abelian groups. Recall, they form definably dual definable subcategories of  $ \Ab=\mathbb Z$-$\Mod$.

Let $M$ be the cyclic abelian group of order $4$ and $N$ its subgroup of order $2$. Of course, there is nothing pure about the inclusion of $N$ in $M$. Yet, it is trivially $\cal L$-pure, for tensoring with a divisible group (with $\cal K$) annihilates every torsion group. 

What's happening on the level of pp formulas? Let $n$ be the generator of $N$ and consider its pp type $p$ in $M$ and $q$ in $N$. Clearly $q\subseteq p$. Consider the formula  $2|x \wedge 2x=0$. Call it $\phi$. Clearly, it is a generator of the type $p$. Let $\psi$ be the formula $2x=0$. This is a generator of the type $q$. So, in particular, they are both $\cal L$-generators of their corresponding types. But they are also  $\cal L$-equivalent, in fact, to $x=0$ (as $\mathbb Z$ is a domain), whence $\psi$ $\cal L$-implies \emph{any} pp formula. So $\psi$ is also an  $\cal L$-generator of $p$, but $\phi$ is not an  $\cal L$-generator  of the type $q$ for the mere reason that it is not contained in it.
\end{exam}

\subsection{Partial/Restricted left purity}

We need a restricted form of purity  in Section \ref{ctblechains} below.

\begin{definition}\label{restricted} 
Let $C$ be a submodule (or even subset) of the module $N$ and $f: N\to M$ a homomorphism.
 \begin{enumerate}[\upshape(1)]
\item The map $f$  is said to be \emph{$\cal L$-pure on $C$} if for every tuple $\bar{c}$ in $C$, its pp type in $N$ is $\langle\cal L\rangle$-equivalent to the  pp type of $f(\bar{c})$ in $M$.
\item $f$ is said  to be \emph{monic on $C$} if $\ker f \cap C=\emptyset$.
\end{enumerate}
\end{definition}

\begin{rem}\textbf{}
\begin{enumerate}[\upshape(1)]
 \item Clearly, $f: N\to M$ is $\cal L$-pure iff it is $\cal L$-pure on $N$.
 \item It is easy to see that the versions of Lemma \ref{FpureDef} and Cor.\ \ref{FpureDefCor} hold for $\cal L$-purity on $C$ when the tuple $\bar{n}$ runs only over $C$ (instead of over all of $N$).
\end{enumerate}
\end{rem}

We are not trying to study this restricted concept in generality here, but rather make use of it in the special context of \Lo-limits below.

\subsection{A special assumption}\label{special} 
The relevance of  $R_R$ being a member of $\cal K$ was pointed out in \cite{HT}. The underlying reason is the following. While, in general, under an $\cal L$-pure homomorphism, a given tuple in the domain and its image have  $\cal L$--equivalent pp types, under the assumption $R_R\in\langle\cal K\rangle$,  the \ quantifier-free parts of those pp types are actually \emph{equal} (or  $\RMod$--equivalent). Note, as $\langle R_R\rangle = \langle\flat_R\rangle $, the clause $R_R\in\langle\cal K\rangle$ is, by elementary duality, equivalent to $_R\sharp\subseteq\langle\cal L\rangle$.

\begin{lem}
 $\sharp$-equivalent types of tuples (in any modules) have identical quantifier-free part.
\end{lem}
\begin{proof} 
 Let $p =\pp_N(\bar{c})$ and $q =\pp_M(\bar{b})$. We are going to show that $p\leq_{\sharp}q$ implies that the quantifier-free part of $q$ is in $p$, i.e., $\qf_M(\bar{b})\subseteq \qf_N(\bar{c})$.
 To this end, embed $N$ in an injective module $E$. Then $\bar{c}$ satisfies $p$ also in $E$. As $E\in \sharp$, $\bar{c}$ satisfies also $q$ in $E$. The quantifier-free part $\qf_M(\bar{b})$ passes down to the submodule $N$, so $\bar{c}$ satisfies $\qf_M(\bar{b})$ in $N$, whence  $\qf_M(\bar{b})\subseteq \qf_N(\bar{c})$, as desired.
\end{proof}

\begin{rem}\label{qfsharpimpl}
 The same proof shows that $\alpha\leq_{\sharp}\phi$ implies $\alpha\leq\phi$ provided $\al$ is quantifier-free. 
\end{rem}
\begin{lem}\label{specialmonic} Suppose $R_R\in\langle\cal K\rangle$ or, equivalently, $_R\sharp\subseteq\langle\cal L\rangle$.

\begin{enumerate}[\upshape(1)]
\item Every map  $\cal L$-pure on a  submodule preserves quantifier-free parts of pp types of tuples from that submodule.
\item Every map  $\cal L$-pure on a  submodule is  monic on it. 
\item $_R\sharp$-pure maps are monomorphisms.
 \end{enumerate}
\end{lem}
\begin{proof}
The last statement is a special case of the second, where the subset in question is the entire module, while the second is a special case of the first, for the (quantifier-free) formula $\bar{x}\, \dot=\, \bar{y}$.

For the first statement, let $f: N\to M$ be $\cal L$-pure on a set $C\subset N$ and let $\bar{c}$ be a tuple in $C$ and $\bar{b}$  the image of $\bar{c}$  in $M$. The assertion now is $\qf_N(\bar{c})= \qf_M(\bar{b})$, which readily follows from the previous lemma.

\end{proof}

\subsection{Right purity: pure epimorphisms}
Dually to the lemma characterizing  $\cal L$-pure monomorphisms, we make the following definition.

\begin{definition}\label{FpureEpiDef}
 A \emph{uniform $\cal L$-pure epimorphism} is a homomorphism $g: N\to M$ (of left $R$-modules) such for every tuple $\bar m$ in $M$ there is a $g$-preimage $\bar n$ in $N$ such that 
their corresponding pp types are $\cal L$-equivalent. We drop the prefix $\cal L$ when it is all of $\RMod$.
\end{definition}

We will encounter uniform pure epimorphisms in Thm.\,\ref{???Existence}. For reference there, we state the following easy fact.

\begin{lem}\label{Lpureepic}
 Every uniformly $\cal L$-pure epimorphic image of a $\cal K$-Mittag-Leffler module is $\cal K$-Mittag-Leffler.
\end{lem}
\begin{proof}
 Let $g: N\to M$ be a uniform $\cal L$-pure epimorphism and $\bar m$ a tuple in $M$. Let $p$ be its pp type. To see that $p$ is finitely $\cal L$-generated, choose a preimage  $\bar n$  with $\cal L$-equivalent pp type, $q$ say. By hypothesis, the pp type $q$  contains a pp formula $\phi$ that $\cal L$-generates it. This formula is contained in  its image $p$ too, and as $p$ and $q$ are $\cal L$-equivalent, the former is $\cal L$-generated by $\phi$ as well.
\end{proof}

\section{Pure separation}\label{separation}

 In the first part I prove  the relativized version of Raynaud and Gruson's separability result (\ref{RG}) from the introduction. Previous, and equivalent versions are \cite[Thm.2.5(ii)]{HT} and  \cite[Prop.3.3]{Tf}. But with a new and more systematic view on relativized purity, it seems justified to include a proof of this result, especially since that very proof goes back again to the unpublished \cite{habil}.\footnote{Where it was the definition that was missing, see the introduction to Section \ref{Fpure}.} Another variant of this phenomenon is considered in the second part. I make a uniform definition.

\begin{definition}
 Let $\kappa$ be a cardinal and $\cal C\subseteq\RMod$. A module $M$ is said to be \emph{$\kappa$-$\cal L$-pure separable by (modules from) $\cal C$} if every subset $A$ of $M$ of cardinality $<\kappa$ is contained in an $\cal L$-pure submodule of $M$ that is a member of $\cal C$.
 
 I say \emph{countably $\cal L$-pure separable} for $\aleph_1$-$\cal L$-pure separable and  \emph{finitely $\cal L$-pure separable} for $\aleph_0$-$\cal L$-pure separable.
\end{definition}

The focus here lies in classes $\cal C$ of $\cal K$-Mittag-Leffler (=$\cal L$-atomic) modules.

\begin{rem}\label{sepML} (1) By the Main Theorem, a module $\cal L$-pure separable by $\cal K$-Mittag-Leffler modules is certainly $\cal K$-Mittag-Leffler.

\noindent
(2)  Since unions of chains of $\cal L$-pure submodules are $\cal L$-pure, it is easy to see that  a module finitely $\cal L$-pure separable by finitely generated  $\cal K$-Mittag-Leffler modules is also countably $\cal L$-pure separable by countably generated  $\cal K$-Mittag-Leffler modules.
\end{rem}

\subsection{Countable separation}
First of all, Corollary \ref{FpureDefCor} implies

\begin{cor} 
$\cal L$-pure submodules  of $\cal K$-Mittag-Leffler modules are $\cal K$-Mittag-Leffler.\qed
\end{cor}

That there are enough such $\cal L$-pure submodules
 can be viewed as a version of the downward L\"owenheim-Skolem Theorem, where `elementary submodule' is weakened to `pure submodule' (which means that the realm of formulas is restricted to pp formulas) and `cardinality of the language' is strengthened to `countably generated' (no matter what that cardinality). Accordingly, a proof of (\ref{RG}) using pp formulas was given in \cite[Lemma 3.11]{habil}\footnote{For another generalization of Mittag-Leffler module, where purities given by classes of finitely presented modules other than $\Rmod$ are considered.} and \cite{ML}.  A version of this proof closer to the original L\"owenheim-Skolem Theorem, \cite[Thm.1.3.27]{P2}, carries over one-to-one to the more general case of  $\cal K$-Mittag-Leffler  modules. 
 
\begin{prop}\label{ctblesep} Every countable subset of a $\cal K$-Mittag-Leffler   module is contained in a countably generated $\cal L$-pure submodule that is  $\cal K$-Mittag-Leffler. 
\end{prop}
\begin{proof} 
Suppose  $M$ is a  $\cal K$-Mittag-Leffler module and $C$ a countable subset.  We are going to construct an $\cal L$-pure submodule $N$ of $M$ that contains $C$ (which then is $\cal K$-Mittag-Leffler).

First, list $C$ as $( c_i )_{i<\omega}$. Then successively extend this  to a list of  tuples  as follows. 
 Starting with  $\bar a_0$ the singleton $c_0$, (or any other tuple containing it), in the $i$th step ($i\geq 0$), choose a pp formula $\phi_i(\bar x_0, \ldots, \bar x_i)$ in $p_i:=\pp_M(\bar a_0, \ldots,\bar a_i)$ that $\cal L$-generates $p_i$. Write $\phi_i$ as $\exists\bar x'_{i+1} \alpha_i(\bar x_0, \ldots, \bar x_i, \bar x'_{i+1})$ with $\alpha_i$ quantifier-free. Choose a witness $\bar a'_{i+1}$ for $\bar x'_{i+1}$  in $M$, concatenate it with $c_{i+1}$ and call the resulting tuple $\bar a_{i+1}$. Accordingly, let $\bar x_{i+1}$ be the tuple of variables obtained from  $\bar x'_{i+1}$ by concatenating one more variable for $c_{i+1}$. We add a dummy variable for that also in $\alpha_i$ so that we may think of its variables as being $\bar x_0, \ldots, \bar x_i, \bar x_{i+1}$ (rather than ending in $\bar x'_{i+1}$). (In case, $\phi_i$ is quantifier-free, $\bar a_{i+1}$ would be the singleton $c_{i+1}$.)

Now, the module, $N$, generated by the $\bar a_i$  is $\cal L$-pure by construction. Indeed,  since $\phi_i$ $\cal L$-generates $p_i$ while its witnesses are inside $N$, $\cal L$-purity is guaranteed for the generators of $N$. But it's a simple exercise to verify the same for any other tuple in $N$ (this was done in \cite[Lemma 1.5]{habil}, cf.\, \cite[Fact 2.4]{PR}). 
\end{proof}

\begin{cor}\label{puresep}
 A module is $\cal K$-Mittag-Leffler iff it is countably $\cal L$-pure separable by countably generated $\cal K$-Mittag-Leffler  modules.\qed
\end{cor}

\subsection{Finite separation} 
Finitely pure separated modules were characterized  in \cite[Rem.3.13]{habil}.\footnote{and called \emph{strong Mittag-Leffler modules}---for purities as mentioned in the previous footnote.} We desribe the $\cal L$-pure variant of this particular kind of $\cal K$-Mittag-Leffler module in the same way now.
 Key is the following.

\begin{lem} \cite[Rem.3.13]{habil}
If the pp type of a tuple $\br a$ in a module $M$ is $\cal L$-generated by a quantifier-free pp formula
, then the submodule generated by $\br a$ is $\cal L$-pure in $M$.

If $M$ is $\cal K$-Mittag-Leffler, the converse is true.
\end{lem}
\begin{proof}
Suppose,  $\pp_{M}(\br a)$ is   $\cal L$-generated by the quantifier-free formula $\al$. Then $\alpha\leq_{\cal L}\phi$ for every $\phi\in\pp_{M}(\br a)$. But being quantifier-free, $\alpha$ is already in $\pp_{A}(\br a)$, so $\al$ makes (iii) of Lemma \ref{Fpurefg} true uniformly, and $A$ is $\cal L$-pure in $M$.

If, conversely, $A$ is $\cal L$-pure in  the $\cal K$-Mittag-Leffler $M$, then $\pp_{M}(\br a)$ is  $\cal L$-generated by a single pp formula $\phi$ and it suffices that this be $\cal L$-implied by a quantifier-free formula $\al\in\pp_{M}(\br a)$.  By purity,  $\phi$ is $\cal L$-implied by some $\psi$ in $\pp_{A}(\br a)$. Write $\psi$ as $\exists\br y \theta(\br x, \br y)$, where $\theta$ is quantifier-free, and pick a witness $\br b\in A$ for $ \theta(\br a, \br y)$. As $\br a$ generates $A$, we can write $\br b$ as a term in $\br a$ or, put in more linear algebraic terminology, there's a matrix $\mtx B$ over the ring such that $\br b=\mtx B\br a$. Let $\al(\br x)$ be the formula $ \theta(\br x, \mtx B\br x)$. It is quantifier-free and in $\pp_{A}(\br a)\subseteq\pp_{M}(\br a)$. 
The  sequence of implications $\al\sim \theta(\br x, \mtx B\br x)\leq \exists\br y \theta(\br x, \br y)\sim\psi\leq_{\cal L}\phi$ concludes the proof.
\end{proof}

\begin{rem}\label{specialfp}
 Applied to the case where a $\cal K$-Mittag-Leffler module $M$ itself is generated by $\br a$, this lemma shows that $\pp_M(\br a)$ is $\cal L$-generated by a quantifier-free formula. 
 
 Together with Remark \ref{qfsharpimpl}, this entails that finitely generated $\sharp$-atomic modules are atomic (that is, $\RMod$-atomic). Thus finitely generated  $R_R$-Mittag-Leffler modules are Mittag-Leffler and hence finitely presented (a fact presumably known since \cite{Goo}).

\end{rem}

\begin{thm} 
A module $M$ is finitely $\cal L$-pure separable by finitely generated  $\cal K$-Mittag-Leffler modules if and only if every finite tuple in $M$ can be extended  to a finite tuple whose pp type  is $\cal L$-generated by a quantifier-free pp formula.

If $R_R\in\langle\cal K\rangle$  (or, equivalently, $_R\sharp\subseteq\langle\cal L\rangle$), then such a 
module $M$ is finitely $\cal L$-pure separable by finitely presented modules.
\end{thm}
\begin{proof}
Suppose $M$ is finitely $\cal L$-pure separable by finitely generated  $\cal K$-Mittag-Leffler modules. 
Then $M$ is $\cal K$-Mittag-Leffler (see the remark at the beginning of this section), and  the second part (the converse) of the lemma applies to prove the direction from left to right.

For the converse, let $\br c$ be a finite subset of $M$ (written as a tuple). By hypothesis, it can be extended to a tuple $\br a$ whose pp type is  $\cal L$-generated by a quantifier-free formula. Call the submodule it generates $C$. By the lemma, $C$ is $\cal L$-pure in $M$. Finally, $C$ is    $\cal L$-atomic by Remark \ref{LLatfg}, hence $\cal K$-Mittag-Leffler.
\end{proof}

\section{Countably generated modules} \label{countablygenerated}
The previous section points to the importance of countably generated $\cal K$-Mittag-Leffler modules as building blocks. The classical ones, i.e., the  countably generated Mittag-Leffler modules (when $\cal K=\ModR$) are precisely the countably generated pure-projectives in $\RMod$ (and this fact is essential ingredient in fact (1) at the very top of this paper). 
For general $\cal K$, this is no longer the case---in fact, it can be seen that it is the case exactly when $\langle\cal K\rangle=\ModR$. In this final section we prove the best approximation 
we were able to find, namely,  Thm.\,\ref{???Existence}, which characterizes countably generated $\cal K$-Mittag-Leffler modules as the uniformly  $\cal L$-pure images of what is called $\cal L$-$\omega$-limits in Section \ref{ctblechains}. 

%


\subsection{Descending chains of pp formulas}
We  first  refine the method of successive realization of pp formulas as follows.

 \begin{rem}\label{Lchain}
 Replacing, in the proof of Proposition \ref{ctblesep}, each $\phi_i$ by its conjunction with $\alpha_{i-1}$, we may (for every $i>0$) ensure that not only $\phi_{i-1}\geq\phi_i$, but even $\alpha_{i-1}\geq \alpha_i$. Note the role of the variables: $\alpha_{i-1}(\bar x_0, \ldots, \bar x_i)\geq \alpha_i(\bar x_0, \ldots, \bar x_i, \bar x_{i+1})$, hence also $\alpha_{i-1}(\bar x_0, \ldots, \bar x_i)\geq \exists\bar x_{i+1} \alpha_i(\bar x_0, \ldots, \bar x_i, \bar x_{i+1})$ (which is equivalent to $\phi_i$). So, not only do we have a descending chain $\phi_0\geq\phi_1\geq\phi_2\geq\ldots$ of pp formulas, but it can be chosen so that the underlying quantifier-free formulas $\alpha_i$ also form a descending chain. 
 Further, as $\phi_{i+1}\in p_{i+1}$, we see that $\exists\bar x_{i+1}\phi_{i+1}$ is contained in the type $p_i$ and therefore $\cal L$-implied by $\phi_i$. So, on top of the `implications'  $\phi_0\geq\phi_1\geq\phi_2\geq\ldots$ and thus,  for all $j>i$, the implications  $\exists \bar x_{i+1} \ldots \exists \bar x_{j}\phi_{j}\leq\phi_{i}$ (in $\RMod$),  we also have the 
converse $\cal L$-implications and therefore $\cal L$-equivalences  $\exists \bar x_{i+1} \ldots \exists \bar x_{j}\phi_{j}\sim_{\cal L}\phi_{i}$ (which may hold only in members of $\cal L$). 

\begin{definition}
 A  chain $\phi_0\geq\phi_1\geq\phi_2\geq\ldots$ as in the above remark is called  an \emph{$\cal L$-chain}.
\end{definition}

A model theory of  $\cal L$-chains may be developed that allows to speak about their free realizations, radicals (in the sense of Herzog, cf.\,\cite{HR}), and the like.
\end{rem}

\begin{lem}\label{existLchain}
Every countable sequence generating a (countably generated)  $\cal K$-Mittag-Leffler module can be arranged as a sequence of tuples satisfying an $\cal L$-chain of pp formulas.
\end{lem}
\begin{proof} Let $N$ be a $\cal K$-Mittag-Leffler module generated by $c_i$  ($i<\omega$). As in Lemma \ref{ctblesep} (and the following remark) we choose tuples $\bar a_0, \bar a_1, \bar a_2, \ldots$ satisfying an $\cal L$-chain of pp formulas, except, this time we choose them only among the generators, i.e., we rearrange the $c_0, c_1, c_2, \ldots$ in such a sequence of tuples. This can be accomplished as follows.

Each time we choose a witness, we express it in terms of the generators. To be specific, when we choose a witness for $\bar x_{i+1}$ in $\phi_i = \exists\bar x_{i+1} \alpha_i(\bar x_0, \ldots, \bar x_i, \bar x_{i+1})$ we  rewrite it in the generators, i.e., we write it as $\mtx B\bar a_{i+1}$, where $\bar a_{i+1}$ is a tuple of generators and $\mtx B$ a matrix over the ring.
Replacing the original $\alpha_i$ by $\alpha_i(\bar x_0, \ldots, \bar x_i, \mtx B\bar x_{i+1})$, we may assume that  $\alpha_i(\bar x_0, \ldots, \bar x_i, \bar x_{i+1})\in \qf(\bar a_0, \ldots, \bar a_i, \bar a_{i+1})$ and $\phi_i=\exists\bar x_{i+1}\alpha_i $.

As before, we also make sure to include the next $c_{i+1}$ from the list of generators in our choice of $\bar a_{i+1}$.  Filling in redundant generators and deleting those that are already listed in $\bar a_0, \ldots, \bar a_i$, we may ensure that they are pairwise disjoint and that their concatenation followed by $\bar a_{i+1}$ constitutes another initial segment of the original list of generators.
\end{proof}

\subsection{Countable chain limits}\label{ctblechains}

\begin{definition}\label{Lpure}\textbf{}
 \begin{enumerate}[\upshape(1)]
\item  Call a chain $M_0 \stackrel{g_0}\to M_1 \stackrel{g_1}\to \ldots \stackrel{g_{k-1}}\to M_k \stackrel{g_k}\to\ldots$ ($k<\omega$)

---by extension also its $\omega$-limit---
 \emph{$\cal L$-pure on images} if the map $g_{i+1}$ is $\cal L$-pure on $\im g_i$ (in the sense of Def.\,\ref{restricted}), for every  $i\in I$.
\item An  \emph{\Lo-limit} 
is an $\omega$-limit of finitely presented modules that is $\cal L$-pure on images.
\item {\rm (Notation)} Given such a chain, we  fix the connecting maps $f_i^{i}$, the identity map on $M_i$,  $f_i^{i+1}=g_i$, and  $f_i^{j+1} = g_{j}\ldots g_i$ for $i < j< \omega$; and  the limit  $M_\omega$ with the canonical maps $f_i: M_i \to M_\omega$ where $f_j f_i^j = f_i$ for all $i\leq j$.

\end{enumerate}
\end{definition}

The definition makes sense for more general direct systems and limits, but this is all we need here.


\begin{lem}\label{LlimitLpure}
Given such a chain  that is $\cal L$-pure on images,  the connecting maps $f_{i+1}^j$ ($j>i$), as well as the canonical map $f_{i+1}$, are   $\cal L$-pure  on $\im g_i$, for all $i$. 

If, in addition, $R_R\in\langle\cal K\rangle$, all these maps are  monic and the  $g_i(M_i)$ are finitely presented.


\end{lem}
\begin{proof} The assertion on the connecting maps follows directly from the transitivity of $\cal L$-equivalence of types. For the second, we have to show that, given $\bar{b}$ in $M_i$, the type $p=\pp_{M_{i+1}}(g_i(\bar{b}))$ and the type $q=\pp_M(f_{i+1}(g_i(\bar{b})))$  are $\cal L$-equivalent. As $p\subseteq q$, we have (even) $q\leq p$ and it remains to verify that, conversely, every $\phi\in q$ is $\cal L$-implied by $p$. 

By Rem.\,\ref{limittruth}, $\phi$ is eventually true in the tail of $f_{i+1}(g_i(\bar{b}))$ starting at $g_i(\bar{b})$. Pick any $M_k$ with $k$ above $i+1$, where $f_{i+1}^k(g_i(\bar{b}))$ satisfies $\phi$. The pp type of this tuple being $\cal L$-equivalent to $p$ by the first part, we see that, in particular, $p$ $\cal L$-implies $\phi$.

The moreover clause follows from Lemma \ref{specialmonic} and  Remark \ref{specialfp}.
 \end{proof}

\begin{rem}\label{gi}
(1) To avoid any misconception, while the \Lo-limit $M_\omega$ is always the direct limit of the chain of monomorphisms $g_0(M_0) \hookrightarrow g_1(M_1) \hookrightarrow g_2(M_2)\ldots$, these images may no longer be finitely presented. 
 
 \noindent
(2)  If $\cal S$ is an ascending chain of coherent modules that is $\cal L$-pure on images, then its direct limit is  a direct limit of an $\cal L$-pure chain of finitely presented modules. Namely, the images of all the maps along the chain are themselves finitely presented and $\cal L$-pure in the next member of the chain. In particular, countably generated left $\cal K$-Mittag-Leffler modules over left coherent rings are `unions' of $\cal L$-pure chains of finitely presented modules (and conversely).
\end{rem}

\begin{lem}\label{LoLatomic}
 Every \Lo-limit is $\cal L$-atomic, hence $\cal K$-Mittag-Leffler.
\end{lem}
\begin{proof}
 Consider a  chain as in \ref{Lpure}(1) where all $M_i$ are finitely presented. Every tuple in its limit   $M_\omega$ is the image under a canonical map of a $g_i$-image in $M_{i+1}$ for some $i$. By the previous lemma, its pp type, $q$, in $M_\omega$ is $\cal L$-equivalent to that of the image tuple in $M_{i+1}$. The latter module being $\cal L$-atomic, we see that also $q$ is $\cal L$-finitely generated.
 \end{proof}

\begin{rem}
The proof shows the same for  a limit  of a chain $\cal L$-pure on images of \emph{arbitrary} $\cal K$-Mittag-Leffler modules.
\end{rem}

Next we establish a connection between \Lo-limits and $\cal L$-chains of pp formulas.

\begin{notat}\label{MPhiDefined}
Suppose  $\Phi=(\phi_i : i<\omega)$ is an $\cal L$-chain with corresponding chain of quantifier-free formulas $(\alpha_i : i<\omega)$  as in Remark \ref{Lchain}.
Every quantifier-free pp formula $\alpha_i$ has a  free realization $(M_i, \bar p_{0}^i, \ldots, \bar p_{i+1}^i)$ with $M_i$ finitely presented and generated by the realizations $\bar p_{0}^i, \ldots, \bar p_{i+1}^i$ of $\alpha_i$. 
 Then $(M_{i+1}, \bar p_0^{i+1}, \ldots, \bar p_{i+1}^{i+1})$ realizes $\alpha_i$, hence  there is a map 
$$g_{i}: (M_i, \bar p_{0}^i, \ldots, \bar p_{i+1}^i)\to (M_{i+1}, \bar p_0^{i+1}, \ldots, \bar p_{i+1}^{i+1}).$$

Note, $(M_i, \bar p_{0}^i, \ldots, \bar p_{i}^i)$ is a free realization of $\phi_i$ for all $i$.

Define the connecting maps $f_i^j$ and the canonical maps $f_i: M_i \to M_\omega$  as in \ref{Lpure}(3). To indicate the origin of this limit, we may denote $M_\omega$ by $M_\Phi$. By the \emph{$i$th tail} in  $M_\Phi$ we mean the tail $\{ p_{i}^j | j\geq i\}$ and denote its limit in  $M_\Phi$ by $\bar q_i$, i.e., 
 $\bar q_i =f_i(\bar p_{i}^i)$ for $i<\omega$ (hence $\bar q_i =f_j(\bar p_{i}^j)$ for any $j\geq i$).

For every $i$, let $\Phi_i:=\{ \exists \bar x_{i+1} \ldots \exists \bar x_{j} \phi_j  : i\leq j<\omega\}$ (this is also a chain of pp formulas, but in finitely many free variables; for $j=i$, the formula is taken to be just $\phi_i$).
\end{notat}

\begin{lem}\label{basicchainstuff}\textbf{}
\begin{enumerate}[\upshape (1)]
\item The tuples $\bar q_0, \bar q_1, \bar q_2, \ldots$ generate the module  $M_\Phi$.
\item For every $i$, the type $\pp_{M_\Phi}(\bar q_0, \ldots, \bar q_i)$ is equivalent to $\Phi_i$.
 \item $\pp_{M_\Phi}(\bar q_0, \ldots, \bar q_i)$ is $\cal L$-generated by  $\phi_i$. 
 \end{enumerate}
\end{lem}
\begin{proof} 
(1) is clear from the construction. (3) follows from (2), since all the formulas in $\Phi_i$ are $\cal L$-equivalent to $\phi_i$ in an  $\cal L$-chain. 
 
 To prove (2), first note that, by  the defining properties of direct limit, $\Phi_i\subseteq \pp_{M_\Phi}(\bar q_0, \ldots, \bar q_i)$, for $f_j( \bar p_{0}^j, \ldots, \bar p_{i}^j)=(\bar q_0, \ldots, \bar q_i)$ for every $j\geq i$ (remember, pp formulas are preserved under homomorphisms).
 
 Now let $\phi=\phi(\bar x_0, \ldots, \bar x_i)\in \pp_{M_\Phi}(\bar q_0, \ldots, \bar q_i)$. Then there must be $j\geq i$ such that $(P_j, \bar p_{0}^j, \ldots, \bar p_{i}^j)\models \phi$. But then,  
 as $(P_j, \bar p_{0}^j, \ldots, \bar p_{j}^j)$ freely realizes $\phi_j$, the formula $\exists \bar x_{i+1} \ldots \exists \bar x_{j} \phi_j $ is freely realized by $(P_j, \bar p_{0}^j, \ldots, \bar p_{i}^j)$ and thus implies  $\phi$, as desired.
\end{proof}

\begin{rem}  If $\sharp\subseteq\langle\cal L\rangle$, then, by Lemma \ref{qfgen},  $\phi_i$  generates the corresponding quantifier-free type in all of $\RMod$.
\end{rem}

\begin{lem}\label{MPhiisLo}
 For every $\cal L$-chain  $\Phi=(\phi_i : i<\omega)$, the limit $M_\Phi$ is an \Lo-limit (hence $\cal K$-Mittag-Leffler). Conversely, every  \Lo-limit  is of the form $M_\Phi$ for some $\cal L$-chain $\Phi$.
\end{lem}
\begin{proof}
 To prove the first assertion, we show $M_\Phi$ is $\cal L$-pure on images. The image of $g_i$  is the submodule of $M_{i+1}$ generated by $\bar p_0^{i+1}, \ldots, \bar p_{i+1}^{i+1}$. That is, we have to verify that the pp type of this tuple of $i+2$ tuples in  $M_{i+1}$ is $\cal L$-equivalent to 
 $\pp_{M_{i+2}}(\bar p_0^{i+2}, \ldots, \bar p_{i+1}^{i+2})$. 
 Now, $(M_{i+1}, \bar p_0^{i+1}, \ldots, \bar p_{i+1}^{i+1})$ is a free realization of $\phi_{i+1}$, hence this formula generates $\pp_{M_{i+1}}(\bar p_0^{i+1}, \ldots, \bar p_{i+1}^{i+1})$. Similarly,  $\phi_{i+2}$ generates the type
  $\pp_{M_{i+2}}(\bar p_0^{i+2}, \ldots, \bar p_{i+2}^{i+2})$, hence $\exists \bar x_{i+2}\phi_{i+2}$ generates $\pp_{M_{i+2}}(\bar p_0^{i+2}, \ldots, \bar p_{i+1}^{i+2})$. But in an $\cal L$-chain, the formulas $\exists \bar x_{i+2}\phi_{i+2}\leq\phi_{i+1}$ are in fact $\cal L$-equivalent, hence so are the generated types, as desired.
  
  For the converse, consider any \Lo-limit $M_\omega$ as in  \ref{Lpure}. Being finitely presented, the module $M_0$, has a tuple of generators, $\bar p^0_0$, that freely realizes a quantifier-free pp formula $\alpha_0(\bar x_0)$. Extend its image $\bar p^1_0 := g_0(\bar p^0_0)$ to a generating tuple $(\bar p^1_0, \bar p^1_1)$ freely realizing a quantifier-free pp formula $\alpha_1(\bar x_0, \bar x_1)$, and so on. As $g_1$ is $\cal L$-pure on $\im g_0$, the submodule of $M_1$ generated by $\bar p^1_0$, we see that its pp type in $M_1$ is $\cal L$-equivalent to the pp type of  $\bar p^2_0 := g_1(\bar p^1_0)$ in $M_2$. Now the former is generated by $\exists \bar x_1\alpha_1(\bar x_0, \bar x_1)$, while the latter is generated by $\exists \bar x_1\exists \bar x_2\alpha_2(\bar x_0, \bar x_1, \bar x_2)$, which are therefore $\cal L$-equivalent too. Continuing like this we see that the chain $\exists \bar x_1\alpha_1\geq \exists \bar x_2\alpha_2\geq \ldots$ forms an $\cal L$-chain, whose limit is $M_\omega$. (One may have to drop $M_0$ to make the indices match, but that is irrelevant in the limit.)
 \end{proof}

 \subsection{Structure of countably generated Mittag-Leffler modules}
\begin{thm}[Characterization]
\label{???Existence} {\hspace{1mm}}

\begin{enumerate}[\upshape(1)]
\item  The following are equivalent for any given module $N$.
\begin{enumerate}[\upshape(i)]
\item $N$  is a countably  generated $\cal K$-Mittag-Leffler module.
\item $N$  is a countably  generated and  some (in fact every) countable sequence of generators  can be arranged as a sequence of tuples satisfying an $\cal L$-chain of pp formulas.
\item $N$ is  a uniform $\cal L$-pure image of an \Lo-limit.
  \end{enumerate}
  
\item  If $R_R\in \langle\cal K\rangle$  or $M_\Phi\in \langle\cal L\rangle$,  the countably generated $\cal K$-Mittag-Leffler modules  are precisely  the \Lo-limits, which, in turn, are unions of $\cal L$-pure chains of  finitely presented modules.
   
\end{enumerate}
\end{thm}
\begin{proof}
(1). As \Lo-limits are countably generated, Lemmas \ref{Lpureepic} and \ref{LoLatomic} show that (iii) implies (i).  By Lemma \ref{existLchain}, (i) implies (ii). For the remaining implication, assume (ii), that is, a countable sequence of tuples $\bar a_0, \bar a_1, \bar a_2, \ldots$ realizing an $\cal L$-chain $\Phi$ as in Remark \ref{Lchain}. Construct the \Lo-limit $M_\Phi$ according to Notation \ref{MPhiDefined}. By Lemma \ref{MPhiisLo}, $M_\Phi$ is an \Lo-limit and we are going to exhibit a uniform $\cal L$-pure epimorphism from it to $N$.

As   $(P_i, \bar p_{0}^i, \ldots, \bar p_{i+1}^i)$ is a free realization of $\alpha_i$, we may send it to $(N, \bar a_{0}, \ldots, \bar a_{i+1})$ by a  map $h_i$.   Clearly,  the maps $h_i$ and $h_j f_i^j$ coincide on the
generators $\bar p_{0}^i, \ldots, \bar p_{i+1}^i$ of $P_i$. Hence $h_i=h_j f_i^j$, and the universal property of direct limit yields a map $h: M_\Phi\to N$ such that $h_i=h f_i$. Note, 
 $h(\bar q_i)=\bar a_i$ for all $i$.

 That $h$ is surjective is immediate from generator considerations. 
 By Lemma \ref{LlimitLpure},  every tuple in $N$ has a preimage under $h$ with $\cal L$-equivalent pp type, whence $h$ is uniformly  $\cal L$-pure. 
 
 (2) Under the extra assumption on $\cal K$, the epimorphism $h$ is a monomorphism by Lemma \ref{specialmonic}, hence an isomorphism. The same holds true if $M_\Phi\in \langle\cal L\rangle$, for then the $\cal L$-equivalent types of, say $\bar m$, in $M_\Phi$ and its image in $N$ are equivalent on $M_\Phi$, so equality of two entries in  $h(\bar m)$ would imply that of the corresponding entries in $\bar m$.
 
 Further, Lemma \ref{LlimitLpure} implies that \Lo-limits are  unions of the finitely presented $g_i(M_i)$.
\end{proof}

 \subsection{The case $\cal L=\flat$} \label{Pruefer}
 In general, the connection between the limit and its uniformly pure image may be quite loose, as we demonstrate in a final note.

Consider $R=\mathbb{Z}$ (or any other domain) and let $\cal L=\flat$ (the torsion-free abelian groups). As was noted in \cite{Tf}, any torsion group is $\sharp$-Mittag-Leffler. Equivalently, every torsion group is $\flat$-atomic, which can also be seen thus:  the formula $rx\doteq 0$ is $\flat$-equivalent to $x\doteq 0$ and thus $\flat$-implies every pp formula. So, whatever torsion formula is contained in the pp type of an element, it  $\flat$-generates its pp type. The same applies to finite tuples  of torsion elements, since the conjunction of the annihilation formulas  $\flat$-implies that the tuple must be the zero tuple and therefore  $\flat$-implies the entire pp type of the tuple.
For the same reason, any two tuples of the same length of torsion elements have $\flat$-equivalent pp types, which is responsible for the fact that uniform $\flat$-purity does not say much. 

Let me  illustrate this further by looking at two different kinds of $\flat$-limit that have a uniform $\flat$-pure image in common. For the first, consider any  quantifier-free chain of the form $r_0x_0\doteq0\geq r_0x_0\doteq0\wedge r_1x_1\doteq0 \geq \ldots$. Let $\tau_i$ denote its $i$th member, i.e., the conjunction $\bigwedge_{j\leq i} r_jx_j\doteq0$. As the variables are completely separate and $0$ is always there as a witness, we see that this chain is a $\flat$-chain (as a matter of fact, even an $\RMod$-chain). Denote the corresponding limit by $T$. It is not hard to see from the construction in Notation \ref{MPhiDefined} that $T$ is the union of the finitely generated groups $P_i=\bigoplus_{j\leq i} \mathbb Z/r_j\mathbb Z$, i.e., $T$ is isomorphic to $\bigoplus_{j< \omega} \mathbb Z/r_j\mathbb Z$. This group is pure-projective and can be mapped onto any other  group generated by a sequence of elements annihilated by those ring elements $r_i$, respectively.  As noted before, any such group is $\flat$-atomic (i.e., $\sharp$-Mittag-Leffler) and hence, by (the proof of) Thm.\ \ref{???Existence}, a uniform  $\flat$--pure image of $T$. 

So \emph{every} image of $T$ is uniformly $\flat$-pure.
Setting $r_i=p^{i+1}$ in the above, we see that the Pr\"ufer group $\mathbb Z_{p^\infty}$ is  a uniform  $\flat$--pure image of $T=\bigoplus_{i< \omega} \mathbb Z/p^{i+1}\mathbb Z$. 

Of course, letting $\alpha_i$ be the formula  $px_0 \doteq 0\ \wedge\  \bigwedge_{j<i} x_j\doteq px_{j+1}$, we obtain a much more meaningful pp chain in $\mathbb Z_{p^\infty}$. As each of these formulas implies torsion in all variables, again they $\flat$-imply just any pp formula, whence these  form a $\flat$-chain. This shows that the Pr\"ufer groups  are  themselves $\flat$-limits.

Namely, a free realization of $\alpha_i$ is given by $(\mathbb Z/p^i \mathbb Z, p^{i-1}c_i, p^{i-2}c_i, \ldots, c_i)$, where $c_i$ is a generator of that group. Obviously, the corresponding map $\rho_i$ from the $i$th free realization to the $(i+1)$st sends $p^j c_i$ to $p^{j+1} c_{i+1}$ and the limit of this chain is (isomorphic to) $\mathbb Z_{p^\infty}$ itself.  This isomorphism is also the uniform $\flat$-pure epimorphism in question. 

Note, the chain of these $\alpha_i$ is obviously also a $\sharp$-chain, whence  the Pr\"ufer groups  are  also $\sharp$-limits.

\bibliographystyle{amsplain}

\end{document}